\documentclass[12pt, reqno]{amsart}

\usepackage{url}
\usepackage{hyperref}
\usepackage{pdfsync}
\usepackage{amssymb}
\usepackage{amsmath}
\usepackage{amsthm}
\usepackage{stmaryrd}
\usepackage{fullpage} 
\usepackage{amscd} 
\usepackage[all]{xy} 
\usepackage{comment} 
\usepackage{mathrsfs} 

\usepackage{enumerate}

\usepackage[alphabetic,backrefs,lite]{amsrefs} 

\usepackage{color}

 \newcommand{\david}[1]{}

\newcommand{\marg}[1]{}

\newcommand{\note}[1]{}

\newcommand{\defi}[1]{\textsf{#1}} 

\newcommand{\kbar}{{\overline{k}}}

\newcommand{\mm}{{\mathfrak m}}

\newcommand{\calC}{{\mathcal C}}

\newcommand{\calF}{{\mathcal F}}

\newcommand{\calM}{{\mathcal M}}

\newcommand{\calO}{{\mathcal O}}

\newcommand{\calS}{{\mathcal S}}

\newcommand{\calV}{{\mathcal V}}

\def\P{\mathbb{P}}
\def\A{\mathbb{A}}

 \DeclareMathOperator{\Hom}{Hom}

\DeclareMathOperator{\Spec}{Spec}

\DeclareMathOperator{\Mod}{Mod}
\DeclareMathOperator{\Open}{Open}
\DeclareMathOperator{\Coh}{Coh}
\DeclareMathOperator{\id}{id}

\DeclareMathOperator{\Th}{th}

\newcommand{\Sch}{\operatorname{\bf Sch}}

\newcommand{\Ab}{\operatorname{Ab}}

\newcommand{\op}{{\operatorname{op}}}
\newcommand{\an}{{\operatorname{an}}}

\DeclareMathOperator{\cosk}{cosk}

\DeclareMathOperator{\sk}{sk}
\DeclareMathOperator{\Simp}{Simp}

\DeclareMathOperator{\AN}{AN} 
\DeclareMathOperator{\Spf}{Spf}
\DeclareMathOperator{\Isoc}{Isoc}
\DeclareMathOperator{\QCoh}{QCoh}
\DeclareMathOperator{\fp}{fp}
\DeclareMathOperator{\g}{g}

\DeclareMathOperator{\Cris}{Cris}

\DeclareMathOperator{\Tot}{Tot}
\DeclareMathOperator{\Mor}{Mor}
\DeclareMathOperator{\Ob}{Ob}
\DeclareMathOperator{\Cosimp}{Cosimp}
\DeclareMathOperator{\ch}{ch}
\DeclareMathOperator{\Ch}{Ch}

\numberwithin{equation}{subsection}
\newtheorem{theorem}[subsection]{Theorem}
\newtheorem{lemma}[subsection]{Lemma}
\newtheorem{corollary}[subsection]{Corollary}
\newtheorem{proposition}[subsection]{Proposition}

\theoremstyle{definition}
\newtheorem{definition}[subsection]{Definition}

\newtheorem{example}[subsection]{Example}

\theoremstyle{remark}
\newtheorem{remark}[subsection]{Remark}

\newcommand{\labelpar}[1]{\refstepcounter{subsection}\label{#1}\thesubsection.}

\newwrite\refs
\openout\refs=\jobname.refs
\makeatletter
\renewcommand\@setref[3]{%
        \ifx#1\relax
                \write\refs{'#3' \thepage\space undefined}%
                \protect \G@refundefinedtrue
                \nfss@text{\reset@font\bfseries ??}%
                \@latex@warning{Reference `#3' on page \thepage\space
                                undefined}%
        \else
                \write\refs{'#3' \thepage\space
                            \expandafter\@secondoftwo#1}%
                \expandafter#2#1\null
        \fi
}
\makeatother

\begin{document}

\title[Cohomological Descent on the Overconvergent site]{Cohomological Descent on the Overconvergent site}

\author{David Zureick-Brown}
\address{Dept. of Mathematics and Computer science, Emory University,
Atlanta, GA 30322 USA}
\thanks{This work was partially supported by a National Defense Science and
Engineering Graduate Fellowship and by a National Security Agency Young Investigator grant.}
\urladdr{http://mathcs.emory.edu/\~{}dzb/}
\date{\today}

\begin{abstract}
We prove that cohomological descent holds for finitely presented crystals on the overconvergent site
with respect to proper or fppf hypercovers.
\end{abstract}

\maketitle


\section{Introduction}
\label{S:introduction}

Cohomological descent is a robust computational and theoretical tool, central to $p$-adic cohomology and its applications. On one hand, it facilitates explicit calculations (analogous to the computation of coherent cohomology in scheme theory via \v{C}ech cohomology); on another, it allows one to deduce results about singular schemes (e.g., finiteness of the cohomology of overconvergent isocrystals on singular schemes \cite{kedlaya:finitenessCoefficients}) from results about smooth schemes, and, in a pinch, sometimes allows one to bootstrap global definitions from local ones (for example, for a scheme $X$ which fails to embed into a formal scheme smooth near $X$, one actually \emph{defines} rigid cohomology via cohomological descent; see \cite{leStum:rigidBook}*{comment after Proposition 8.2.17}).
\vspace{4pt}

The main result of the series of papers \cite{ChiarellottoT:etaleRigidDescent, Tsuzuki:properRigidDescent, Tsuzuki:rigidDescent} is that cohomological descent for the rigid cohomology of overconvergent isocrystals holds with respect to both flat and proper hypercovers. The barrage of choices in the definition of rigid cohomology is burdensome and makes their proofs of cohomological descent very difficult, totaling to over 200 pages. Even after the main cohomological descent theorems \cite{ChiarellottoT:etaleRigidDescent}*{Theorems 7.3.1 and 7.4.1} are proved one still has to work a bit to get a spectral sequence \cite{ChiarellottoT:etaleRigidDescent}*{Theorem 11.7.1}. Actually, even to state what one means by cohomological descent (without a site) is subtle.
\vspace{5pt}

The situation is now more favorable. More than 25 years after Berthelot's seminal papers \cite{Berthelot:rigidFirst, berthelot:finitude, berthelot:dualite}, key foundational aspects have now been worked out. Le Stum's recent advance \cite{leStum:site} is the construction of an `overconvergent site' \cite{leStum:site} which gives an alternative, equivalent definition of rigid cohomology as the cohomology of the structure sheaf of a ringed site $(X_{\text{AN}^{\dagger}}, \calO_X^{\dagger})$ (and also of course an equivalence between the category of overconvergent isocrystals on $X$ and the category of finitely presented $\calO_{X}^{\dagger}$-modules). This formalism is the correct setting for many problems; for instance, \cite{leStum:constructibleNabla} uses the overconvergent site to develop a theory of constructible $\nabla$-modules on curves.
\vspace{4pt}

More applications are expected. And indeed, the main result of this paper
is the application of the abstract machinery of \cite{SGA4:II}*{Expos\'e Vbis and VI} to the overconvergent site to give a short proof of the following.

\begin{theorem}
\label{T:mainCohDescentTheoremPrelude}

Cohomological descent for locally finitely presented modules on the overconvergent site holds with respect to
\begin{itemize}
\item [(i)] fppf hypercovers of varieties, and

\item [(ii)] proper hypercovers of varieties.
\end{itemize}

\end{theorem}

\begin{remark}
The proof of Theorem \ref{T:mainCohDescentTheoremPrelude} is not merely a formal consequence of the techniques of \cite{SGA4:II}*{Expos\'e Vbis and VI}. Cohomological descent for abelian sheaves on the \'etale site with respect to smooth hypercovers is simply \v{C}ech theory; see Theorem \ref{T:CDexamples} (i). In the overconvergent setting, an \'etale surjection is not a covering, and hence \v{C}ech theory does not apply. Another technical difficulty is that one cannot check triviality of an overconvergent sheaf $\calF \in \AN^{\dagger} X$ by restricting to points of $X$, so that the template of the proof of proper cohomological descent for \'etale cohomology therefore does not apply to overconvergent cohomology and a new argument is needed.
\end{remark}

\begin{remark}
We emphasize that, while similar looking results appear in the literature (see e.g. Section \ref{S:etaleCohDescent} and Lemma \ref{L:CDblowUp}), it takes additional work to deduce from these corresponding theorems on the overconvergent site. Moreover, new ideas -- for instance, the use of Raynaud-Gruson's theorem on `flattening stratifications' \cite{GrusonR:flatification}*{Th\'eor\`em 5.2.2}, and le Stum's main theorems (Proposition \ref{P:coverings}) -- greatly simplify and extend the generality of our proof.  
\end{remark}

Finally, in light of the central role that one expects le Stum's work to play in the future development of rigid cohomology, we note that various ingredients of our proofs are useful lemmas which facilitate computations on the overconvergent site; see for instance Lemma \ref{L:tubeIsoDescent}.

\subsection{Applications}
\label{SS:applications}

We highlight a few direct applications of our main theorem.

\begin{enumerate}

\item (Spectral sequence.) By le Stum's comparison theorems between rigid and overconvergent cohomology \cite{leStum:site}*{Corollary 3.5.9}, we obtain a spectral sequence (see Remark \ref{R:augmentedCohomology}) computing rigid cohomology; this gives a shorter proof of Theorem 11.7.1 of \cite{ChiarellottoT:etaleRigidDescent}. While this corollary of our work and \cite{ChiarellottoT:etaleRigidDescent, Tsuzuki:properRigidDescent} are similar, the main results are independent and cannot be deduced from one another. 

\item (Overconvergent de Rham-Witt cohomology.) \cite{DavisLZ:Rham}  proves directly that overconvergent de Rham-Witt cohomology agrees with classical rigid cohomology for smooth affine varieties, and a long argument with dagger spaces is needed to deduce agreement with general rigid cohomology. Use of the overconvergent site and  Theorem \ref{T:mainCohDescentTheoremPrelude} simplifies the globalization argument (this will appear in future work \cite{MeD:RhamWittStacks}).

\item (Rigid cohomology for stacks.) Motivated by applications to geometric Langlands, Kedlaya proposed the problem of generalizing rigid cohomology to stacks. There are three approaches -- le Stum's site gives a direct approach realized in the author's thesis \cite{Brown:RigidStacks}); \cite{DavisLZ:Rham} overconvergent de Rham-Witt complex gives an alternative, explicit and direct construction. Theorem \ref{T:mainCohDescentTheoremPrelude} gives a third approach and a direct comparison of the first two approaches; moreover theorem \ref{T:mainCohDescentTheoremPrelude}  also gives a direct proof that the rigid cohomology of a stack is finite dimensional, and allows one to make various constructions (e.g., to define a Gysin map).

\end{enumerate}

\subsection{Acknowledgements}
\label{SS:acknowledgements}

We would like to thank Bjorn Poonen, Brian Conrad, Arthur Ogus, Bernard le Stum, Bruno Chiarellotto, and Anton Geraschenko for many useful conversations and encouragement. We also remark that the \'etale case was part of the author's thesis \cite{Brown:RigidStacks}.

\subsection{Organization of the paper}
\label{SS:organization}

This paper is organized as follows. In section \ref{S:conventions} we recall notation. In section \ref{S:background} we review the machinery of cohomological descent. In section \ref{S:definitions} we recall the construction of the overconvergent site of \cite{leStum:site}. In section \ref{S:overCohDescent} we prove Theorem \ref{T:mainCohDescentTheoremPrelude}, first in the case of Zariski hypercovers, then in the case of fppf hypercovers, and finally for proper hypercovers.

\section{Notation and conventions}
\label{S:conventions}

Throughout $K$ will denote a field of characteristic 0 that is complete with respect to a non-trivial non-archimedean valuation with valuation ring $\calV$, whose maximal ideal and residue field we denote by $\mm$ and $k$. We denote the category of schemes over $k$ by $\Sch_k$. We define an \defi{algebraic variety over} $k$ to be a scheme such that there exists a \emph{locally finite} cover by schemes of finite type over $k$ (recall that a collection $\calS$ of subsets of a topological space $X$ is said to be locally finite if every point of $X$ has a neighborhood which only intersects finitely many subsets $X \in \calS$). Note that we do not require an algebraic variety to be reduced, quasi-compact, or separated.
\vspace{3pt}

\textbf{Formal Schemes}: As in \cite{leStum:site}*{1.1} we define a formal $\calV$-scheme to be a locally topologically finitely presented formal scheme $P$ over $\calV$, i.e., a formal scheme $P$ with a locally finite covering by formal affine schemes $\Spf A$, with $A$ topologically of finite type (i.e., a quotient of the ring $\calV\{T_1,\cdots,T_n\}$ of convergent power series by an ideal $I + \mathfrak{a}\calV\{T_1,\cdots,T_n\}$, with $I$ an ideal of $\calV\{T_1,\cdots,T_n\}$ of finite type and $\mathfrak{a}$ an ideal of $\calV$). This finiteness property is necessary to define the `generic fiber' of a formal scheme (see \cite{Berkovich:vanishingFormalI}*{Section 1}).

We refer to \cite{EGAI}*{1.10} for basic properties of formal schemes. The first section of \cite{Berkovich:contractiblity} is another good reference; a short alternative is \cite{leStum:site}*{Section 1}, which contains everything we will need.
\vspace{3pt}

\textbf{$K$-analytic spaces}: We refer to \cite{Berkovich:nonArchEtaleCoh} (as well as the brief discussion in \cite{leStum:site}*{4.2}) for definitions regarding $K$-analytic spaces. As in \cite{leStum:site}*{4.2}, we define an \defi{analytic variety} over $K$ to be a locally Hausdorff topological space $V$
together with a maximal affinoid atlas $\tau$ which is locally defined by \emph{strictly} affinoid algebras (i.e., an algebra $A$ is strict if it is a quotient of a Tate algebra $K\{T_1,\cdots, T_n\}$) and denote by $\calM(A) $ the \defi{Gelfand spectrum} of an affinoid algebra $A$. Moreover, recall that a $K$-analytic space is said to be \defi{good} if every point has an open affinoid neighborhood.
\vspace{3pt}

\textbf{Topoi}: We follow the conventions of \cite{SGA4:I} (exposited in \cite{leStum:site}*{4.1}) regarding sites, topologies, topoi, and localization. When there is no confusion we will identify an object $X$ of a category with its associated presheaf $h_X\colon Y \mapsto \Hom(Y,X)$. For an object $X$ of category $C$ we denote by $C_{/X}$ the \defi{comma category}; objects of $C_{/X}$ are morphisms $Y \to X$, and morphisms are commutative diagrams. For a topos $T$ we denote by $\mathbb{D}_+(T)$ the derived category of bounded below complexes of objects of $\Ab T$. Often (in this paper) a morphism $(f^{-1},f_*)\colon (T,\calO_T) \to (T',\calO_{T'})$ of ringed topoi will satisfy $f^{-1}\calO_{T'} = \calO_T$, so that there is no distinction between the functors $f^{-1}$ and $f^*$; in this case, we will write $f^*$ for both. Finally, we note that the category $\Mod_{\fp} \calO_T$ of $\calO_T$-modules which locally admit a finite presentation $\bigoplus_{i = 1}^n \calO_T \to \bigoplus_{i = 1}^m \calO_T \to M$, is generally \emph{larger} than $\Coh \calO_T$, since in general the sheaf of rings $\calO_T$ is not itself coherent.

\section{Background on cohomological descent}
\label{S:background}

Here we recall the definitions and facts about cohomological descent that we will need. The standard reference is \cite{SGA4:II}*{Expos\'e Vbis and VI}; some alternatives are Deligne's paper \cite{Deligne:Hodge3} and Brian Conrad's notes \cite{Conrad:cohDescent}; the latter has a lengthy introduction with a lot of motivation and gives more detailed proofs of some theorems of \cite{SGA4:II} and \cite{Deligne:Hodge3}.
\vspace{3pt}

\labelpar{}
We denote by $\Delta$ the simplicial category whose objects the are the sets $[n] := \{0,1,\ldots,n\}$, $n \geq 0$, and whose morphisms are monotonic maps of sets $\phi\colon [n] \to [m]$ (i.e., for $i \leq j$, $\phi(i) \leq \phi(j)$). We define the augmented simplicial category to be $\Delta^{+} := \Delta \cup \{\varnothing\}$. A \defi{simplicial} (resp. \defi{augmented simplicial}) object $X_{\bullet}$ of a category $C$ is a functor $X_{\bullet}\colon \Delta^{\op} \to C$ (resp. $X_{\bullet}\colon (\Delta^+)^{\op} \to C$); one denotes by $X_n$ the image of $n$ under $X_{\bullet}$. We will typically write an augmented simplicial object as $X_{\bullet} \to X_{-1}$, where $X_{\bullet}$ is the associated simplicial object. A morphism between two simplicial or augmented simplicial objects is simply a natural transformation of functors. We denote these two categories by $\Simp C$ and $\Simp^+ C$.

Similarly, we define the truncated simplicial categories $\Delta_{\leq n} \subset \Delta$ and $\Delta_{\leq n}^+ \subset \Delta^+$ to be the full subcategories consisting of objects $[m]$ with $m \leq n$ (with the convention that $[-1] = \varnothing$). We define the category $\Simp_n C$ of $n$-\defi{truncated simplicial objects} of $C$ to be the category of functors $X_{\bullet}\colon \Delta_{\leq n}^{\op} \to C$ (and define $\Simp_n^+ C$ analogously).
\vspace{7pt}

\labelpar{}
Any morphism $p_0 \colon X \to Y$ in a category $C$ gives rise to an augmented simplicial object $p\colon X_{\bullet} \to Y$ with $X_n$ the fiber product of $n+1$ many copies of the morphism $p_0$; in this case we denote by $p_n$ the morphism $X_n \to Y$ and by $p_i^j$ the $j^{\text{th}}$ projection map $X_{i} \to X_{i-1}$ which forgets the $j^{\text{th}}$
component.
\vspace{7pt}

\labelpar{}
This last construction is right adjoint to the forgetful functor $X_{\bullet} \mapsto (X_0 \to X_{-1})$ from $\Simp^+ C \to \Simp_{\leq 0}^+ C$. We can generalize this point of view to construct an augmented simplicial object out of an $n$-truncated simplicial object as follows. We first define the $n$-skeleton functor
\[
\sk_n\colon \Simp C \to \Simp_{\leq n} C
\]
by sending $X_{\bullet}\colon \Delta^{\op} \to C$ to the composition $\sk_n(X_{\bullet})\colon \Delta_{\leq n}^{\op} \subset \Delta^{\op} \to C$. We define an augmented variant
\[
\sk_n\colon \Simp^+ C \to \Simp^+_{\leq n} C
\]
similarly, which we also denote by $\sk_n$. When $C$ admits finite limits the functor $\sk_n$ has a right adjoint $\cosk_n$ \cite{Conrad:cohDescent}*{Theorem 3.9}, which we call the $n$-\defi{coskeleton}. When we denote a truncated augmented simplicial object as $X_{\bullet} \to Y$, we may also write $\cosk_n(X_{\bullet}/Y) \to Y$ to denote $\cosk_n(X_{\bullet} \to Y)$ (so that $\cosk_n(X_{\bullet}/Y)$ is a simplicial object).
\vspace{7pt}

\labelpar{}
\label{p:fiberedSite}
When $C$ is a site we promote these notions a bit. The codomain fibration, i.e., the fibered category $\pi\colon \Mor C \to C$ which sends a morphism $ X \to Y \in \Ob \left( \Mor C \right)$ to its target $Y$ is a prestack if and only if $C$ is subcanonical (i.e., representable objects are sheaves), and a stack if every $F \in \widetilde{C}$ is representable (equivalently, if the Yoneda embedding $C \to \widehat{C}$ induces an isomorphism $C \to \widetilde{C}$). The fibers are the comma categories $C_{/X}$, and the site structure induced by the projection $C_{/X} \to C$ makes $\pi$ into a \defi{fibered site} (i.e., a fibered category with sites as fibers such that for any arrow $f\colon X \to Y$ in the base, any cartesian arrow over $f$ induces a functor $C_{/X} \to C_{/Y}$ which is a continuous morphism of sites; see \cite{SGA4:II}*{Expos\'e VI}). For a simplicial object $X_{\bullet}$ of $C$, the 2-categorical fiber product $\Delta^{\op} \times_C \Mor C \to \Delta^{\op}$ also is a fibered site; to abusively notate this fiber product as $X_{\bullet}$ will cause no confusion. We will call a site fibered over $\Delta^{\op}$ a \defi{simplicial site}. We define a morphism of fibered sites below \ref{p:fiberedTopos}.
\vspace{7pt}

\labelpar{}
\label{p:fiberedTopos}
Associated to any fibered site is a \defi{fibered topos}; we explicate this for the fibered site $X_{\bullet} \to \Delta^{\op}$ associated to a simplicial object $X_{\bullet}$ of a site $C$. We define first the \defi{total site} $\Tot X_{\bullet}$ to be the category $X_{\bullet}$ together with the smallest topology such that for every $n$, the inclusion of the fiber $X_n$ into $X_{\bullet}$ is continuous. The \defi{total topos} of $X_{\bullet}$ is then defined to be the category $\widetilde{X_{\bullet}}$ of sheaves on $\Tot X_{\bullet}$. We can define a morphism of fibered sites to be a morphism of fibered categories which induces a continuous morphism of total sites.

For $F_{\bullet} \in \widetilde{X_{\bullet}}$ denote by $F_n$ the restriction of $F_{\bullet}$ to $X_n$; as usual for any cartesian arrow $f$ over a map $d' \to d$ in $\Delta^{\op}$ one has an induced map $f^*F_d \to F_{d'}$ and as one varies $d' \to d$, these maps enjoy a cocycle compatibility. The total topos $\widetilde{X_{\bullet}}$ is equivalent to the category of such data. One can package this data as sections of a fibered topos $T_{\bullet} \to \Delta^{\op}$ (with fibers $T_n = \widetilde{C_n}$), i.e., a fibered category whose fibers are topoi such that cartesian arrows induce morphisms of topoi (or rather, the pullback functor of a morphism of topoi) on fibers. The total topos $\widetilde{X_{\bullet}}$ is then equivalent to the category of sections of $T_{\bullet} \to \Delta^{\op}$. When the topology on each fiber $X_n$ is subcanonical, the topology on $\Tot X_{\bullet}$ also is subcanonical and the inclusion ${X_{\bullet}} \subset T_{\bullet}$ of fibered sites (where one endows each fiber $T_n$ of the fibered topos $T_{\bullet}$ with its canonical topology) induces an equivalence of categories
of total topoi.
\vspace{7pt}

\labelpar{}
Let $p_0\colon X \to Y$ be a morphism of presheaves on a site $C$. As before, this gives rise to an augmented simplicial presheaf $p\colon X_{\bullet} \to Y$. Denoting by $\widehat{C}$ the category of presheaves on $C$, we may again promote $X_{\bullet}$ to a fibered site and study its fibered topos as in \ref{p:fiberedSite} above. Indeed, Yoneda's lemma permits one to consider the fibered site $\Mor' \widehat{C} \to \widehat{C}$ (where $\Mor' \widehat{C}$ is the subcategory of $\Mor \widehat{C}$ whose objects are arrows with source in $C$ and target in $\widehat{C}$), and again the 2-categorical fiber product $\Delta^{\op} \times_{\widehat{C}} \Mor' \widehat{C}$ is a fibered site. We also remark that passing to the presheaf category allows one to augment any simplicial object in $C$ by sending $\varnothing$ to the final object of $\widehat{C}$ (which is represented by the punctual sheaf).
\vspace{7pt}

\labelpar{}
\label{p:morphismTotalTopoi}
A morphism $f\colon X_{\bullet} \to Y_{\bullet}$ of simplicial sites induces a morphism $(f^*,f_*)\colon \widetilde{X_{\bullet}} \to \widetilde{Y_{\bullet}}$ of their total topoi; concretely, the morphisms of topoi $(f_{n}^*,f_{n*})\colon \widetilde{X_n} \to \widetilde{Y_n}$ induce for instance a map $\{F_n\} \mapsto \{f_{n*}F_n\}$ which respects the cocycle compatibilities.

To an augmented simplicial site $p\colon X_{\bullet} \to S$ one associates a morphism $(p^*, p_*)\colon \widetilde{X_{\bullet}} \to \widetilde{S}$ of topoi as follows. The pullback functor $p^*$ sends a sheaf of sets $\calF$ on $S$ to the collection $\{p_n^*\calF\}$ together with the canonical isomorphisms $p_{n+1}^{j*}p_n^*\calF_n \cong p_{n+1}^*\calF$ induced by the canonical isomorphism of functors $p_{n+1}^{j*}\circ p_n^* \cong p_{n+1}^*$ associated to the equality $p_{n+1} = p_n \circ p_{n+1}^j$. Its right adjoint $p_*$ sends the collection $\{\calF_n\}$ to the equalizer of the cosimplicial sheaf
\begin{equation}
\label{e:cosimplicial}
\xymatrix{
\cdots p_{(n-1)*}\calF_{n-1}
\ar@{=>}[r]
& p_{n*}\calF_{n}
\ar@3{->}[r]
& p_{(n+1)*}\calF_{n+1} \cdots
}
\end{equation}
where the $n+2$ maps between $p_{n*}\calF_n$ and $p_{n+1*}\calF_{n+1}$ are the pushforwards $p_{n*}$ of the adjoints $\calF_n \to p_{n+1*}^j\calF_{n+1}$ to $p_{n+1}^{j*}\calF_n \to \calF_{n+1}$ (using the equality $p_{(n+1)*} = p_{n*} \circ p_{n+1*}^j $). It follows from an elementary manipulation of the simplicial relations that the equalizer of \ref{e:cosimplicial} only depends on the first two terms; i.e., it is equal to the equalizer of
\[
\xymatrix{
p_{0*}\calF_0
\ar@<3pt>[r]\ar@<-3pt>[r] &
p_{1*}\calF_{1}
}.
\]

One can of course derive these functors, and we remark that while, for an augmented simplicial site $p\colon X_{\bullet} \to S$ and an abelian sheaf $\calF \in \Ab X_{\bullet}$, the sheaf $p_*\calF$ only depends on the first two terms of the cosimplicial sheaf of \ref{e:cosimplicial}, the cohomology $\mathbb{R}p_*\calF$ depends on the entire cosimplicial sheaf. Finally, we note the standard indexing convention that for a complex $\calF_{\bullet,\bullet}$ of sheaves on $X_{\bullet}$, for any $i$ we have that $\calF_{\bullet,i} \in \Ab X_{\bullet}$.

\begin{example}[\cite{Conrad:cohDescent}*{Examples 2.9 and 6.7}]
\label{E:constantSimplicial}

Let $S \in C$ be an object of a site and let $q\colon S_{\bullet} \to S$ be the constant augmented simplicial site associated to the identity morphism $\id\colon S \to S$. The total topos $\widetilde{S_{\bullet}}$ is then equivalent to the category $\Cosimp \widetilde{S} = \Hom(\Delta,\widetilde{S})$ of co-simplicial sheaves on $S$ and $\Ab(S_{\bullet})$ is equivalent to $\Cosimp \Ab(S)$.

(i) It is useful to consider the functor
\[
\ch\colon \Cosimp \Ab(S) \to \Ch_{\geq 0}(\Ab(S))
\]
to the category of chain complexes concentrated in non-negative degree which sends a cosimplicial sheaf to the chain complex whose morphisms are given by alternating sums of the simplicial maps. The direct image functor $q_*$ is then given by
\[
\calF_{\bullet} \mapsto H^0(\ch \calF_{\bullet}) = \ker(\calF_0 \to \calF_1).
\]

Let $I_{\bullet} \in \Ab S_{\bullet}$. Then $I_{\bullet}$ is injective if and only if $\ch I_{\bullet}$ is a split exact complex of injectives (this is a mild correction of \cite{Conrad:cohDescent}*{Corollary 2.13}). Furthermore, for $I_{\bullet} \in \Ab S_{\bullet}$ injective, the natural map
\[
\mathbb{R}q_* I_{\bullet} := q_* I_{\bullet} \to \ch I_{\bullet}
\]
is a quasi-isomorphism and thus $\mathbb{R}^iq_*I_{\bullet} = H^i(\ch I_{\bullet})$. One concludes by \cite{Hartshorne:AG}*{Theorem 1.3A} that the collection of functors $\calF_{\bullet} \mapsto H^i(\ch \calF_{\bullet})$ (the $i^{\Th}$ homology of the complex $\ch \calF_{\bullet}$) forms a universal $\delta$ functor and thus that $\mathbb{R}^iq_* \calF_{\bullet} \cong H^i(\ch(\calF_{\bullet}))$.

(ii) Actually, a mildly stronger statement is true: for an injective resolution $\calF_{\bullet} \to I_{\bullet, \bullet}$ (where $I_{\bullet, i} \in \Cosimp \Ab(S)$), one can show that the map $\ch \calF_{\bullet} \to \ch I_{\bullet, \bullet}$ induces a quasi-isomorphism $\ch \calF_{\bullet} \to \Tot \ch I_{\bullet, \bullet}$, where $\Tot$ is the total complex constructed by collapsing the double complex $\ch I_{\bullet, \bullet}$ along the diagonals. On the other hand the natural map $\mathbb{R}q_* \calF_{\bullet} := q_* I_{\bullet, \bullet} \to \Tot \ch I_{\bullet, \bullet}$ is an isomorphism.
Putting this together we see that the natural map $\mathbb{R}q_* \calF_{\bullet} \to \ch \calF_{\bullet}$ is a quasi-isomorphism. Moreover, the diagram
\[
\xymatrix{
\ch \calF_{\bullet} \ar[r] & \Tot \ch r_* I_{\bullet, \bullet}&\\
q_*\calF_{\bullet} \ar[r] \ar[u] & \mathbb{R}q_* \calF_{\bullet} \ar[u] \ar@{=}[r]& q_*I_{\bullet, \bullet}
}
\]
commutes, so that the natural map $q_* \calF_{\bullet} \to \mathbb{R}q_* \calF_{\bullet}$ is an isomorphism if and only if $q_* \calF_{\bullet} \to \ch \calF_{\bullet}$ is an isomorphism.

(iii) We note a final useful computation. Let $I_{\bullet,\bullet} \in \mathbb{D}_+(S_{\bullet})$ a complex of injective sheaves. Define $I_{-1,n} = \ker \left(\ch I_{\bullet, n}\right)$; by \cite{stacks-project}*{\href{http://math.columbia.edu/algebraic_geometry/stacks-git/locate.php?tag=015Z}{015Z}} (noting that since $q$ is a morphism of topoi, $q^*$ is an exact left adjoint to $q_*$) this is an injective sheaf. Then the hypercohomology of $I_{\bullet,\bullet}$ is simply (by definition) $\mathbb{R}q_*\left(I_{\bullet,\bullet}\right):= q_*\left(I_{\bullet,\bullet}\right) = I_{-1,\bullet}$.

\end{example}

\begin{remark}
\label{R:augmentedCohomology}

Let $p\colon X_{\bullet} \to S$ be an augmented simplicial site, and let $\calF_{\bullet} \in \widetilde{X_{\bullet}}$ be a sheaf of abelian groups. Using Example \ref{E:constantSimplicial}, we can clarify the computation of the cohomology $\mathbb{R}p_*\calF_{\bullet}$ via the observation that the associated map of topoi factors as
\[
\widetilde{X_{\bullet}} \xrightarrow{r} \widetilde{S_{\bullet}} \xrightarrow{q} \widetilde{S},
\]
where $r_*\calF_{\bullet}$ is the cosimplicial sheaf given by Equation \ref{e:cosimplicial}. Therefore, to compute $\mathbb{R}p_*\calF_{\bullet}$ we first study $\mathbb{R}r_*\calF_{\bullet}$.

Set $\calF_{-1} = p_*\calF_{\bullet} = \ker \ch r_* \calF_{\bullet}$. Viewing $\calF_{-1}$ as a complex concentrated in degree 0, we can consider the morphism of complexes $\calF_{-1} \to \ch r_* \calF_{\bullet}$. When $\calF_{\bullet}$ is injective, $r_*\calF_{\bullet}$ also is injective by \cite{stacks-project}*{\href{http://math.columbia.edu/algebraic_geometry/stacks-git/locate.php?tag=015Z}{015Z}}; applying the description of injective objects of $\Ab(S_{\bullet})$ of Example \ref{E:constantSimplicial} (i) to $\ch r_*\calF_{\bullet}$ we conclude that the map of complexes $\calF_{-1} \to \ch r_*\calF_{\bullet}$ is a quasi-isomorphism when $\calF_{\bullet}$ is injective.

Let $\calF_{\bullet} \to I_{\bullet, \bullet}$ be an injective resolution of $\calF_{\bullet}$. Then one gets a commutative diagram of chain complexes
\begin{equation}
\label{M:simplicialCohomologyDiagram}
\xymatrix{
I_{-1,\bullet} \ar[r] & \ch r_* I_{\bullet, \bullet}\\
\calF_{-1} \ar[r] \ar[u] & \ch r_* \calF_{\bullet} \ar[u]
};
\end{equation}
we can alternatively view Diagram \ref{M:simplicialCohomologyDiagram} as a double complex, indexed so that the sheaf $\calF_{-1}$ lives in bi-degree $(-1,-1)$. By the remark at the end of the preceding paragraph, all rows of Diagram \ref{M:simplicialCohomologyDiagram} except the bottom are quasi-isomorphisms;
the columns are generally not quasi-isomorphisms (since $\ch r_* \calF_{\bullet}$ is not exact in positive degree). Now we compute that
\begin{equation}
\label{eq:simpCohcomputation}
\mathbb{R}p_* \calF_\bullet := p_* I_{\bullet,\bullet} = I_{-1,\bullet}.
\end{equation}
The output $I_{-1,\bullet}$ is quasi-isomorphic to the total complex $\Tot \ch r_* I_{\bullet, \bullet}$
(given by collapsing the diagonals); since the $i^{\, \Th}$ column of the double complex $\ch r_* I_{\bullet,\bullet}$ computes $\mathbb{R}p_{i*}\calF_i$, there is an $E_1$-spectral sequence
\begin{equation}
\label{spectralSequence}
\mathbb{R}^jp_{i*}\calF_i = H^j(p_{i*} I_{\bullet, \bullet} ) \Rightarrow H^{i+j}(\Tot \ch r_* I_{\bullet, \bullet}) \cong H^{i+j}(I_{-1,\bullet}) \cong \mathbb{R}^{i + j}p_*\calF_{\bullet},
\end{equation}
where the last isomorphism is the $(i+j)^{\Th}$ homology of Equation \ref{eq:simpCohcomputation}.

\end{remark}

Our later computations will rely on the following degenerate case of the preceding remark.

\begin{corollary}
\label{C:exactImpliesCD}

Let $p\colon X_{\bullet} \to S$ be an augmented simplicial site. Then the following are true.

\begin{itemize}
\item [(i)] Let $\calF_{\bullet} \in \Ab X_{\bullet}$ be a sheaf of abelian groups. Suppose that for $i \geq 0$ and $j > 0$, one has $\mathbb{R}^jp_{i*}\calF_i = 0$. There is a quasi-isomorphism $\mathbb{R}p_* \calF_{\bullet} \cong \ch r_* \calF_{\bullet}$.

\item [(ii)] Let $\calF \in \widetilde{S}$ be an abelian sheaf such that for $i \geq 0$ and $j > 0$, $\mathbb{R}^jp_{i*}p_i^*\calF = 0$, such that $\ch r_* p^*\calF$ is exact in positive degrees, and such that the adjunction $\calF \to \ker(\calF_0 \to \calF_1)$ is an isomorphism. Then $\calF \to \mathbb{R}p_*p^*\calF$ is a quasi-isomorphism.

\end{itemize}

\end{corollary}

\begin{proof}

The second claim is a special case of the first claim. By \cite{Conrad:cohDescent}*{Lemma 6.4}, for any $i,j$, the sheaf $I_{i,j}$ is injective, and thus the $i^{\, \Th}$ column of $I_{\bullet,\bullet}$ is an injective resolution of $\calF_i$. For $i \geq 0$, the $i^{\, \Th}$ column of $r_* I_{\bullet,\bullet}$ is the complex $\mathbb{R}p_{i*}\calF_i$. Thus by hypothesis the complex $r_{i*}\calF_i \to r_{i*} I_{i, \bullet}$ is exact and it follows that the map
\[
\ch r_* \calF_{\bullet} \to \Tot \ch r_* I_{\bullet, \bullet} =: \mathbb{R}p_* \calF_{\bullet}
\]
of Diagram \ref{M:simplicialCohomologyDiagram} is a quasi-isomorphism.

\end{proof}

\begin{remark}

Let $f\colon X_{\bullet} \to Y_{\bullet}$ be a map of simplicial sites, $\calF_{\bullet} \in \widetilde{X_{\bullet}}$ be a sheaf of abelian groups, and suppose that for every $i \geq 0$, the natural map $f_{i*} \calF_i \to \mathbb{R}f_{i_*}\calF_i$ is a quasi-isomorphism. Then the strategy used in the proof of Corollary \ref{C:exactImpliesCD} generalizes to prove that the natural map $f_*\calF_{\bullet} \to \mathbb{R}f_*\calF_{\bullet}$ is a quasi-isomorphism.

\end{remark}

Finally, we arrive at the main definition.

\begin{definition}
\label{D:cohomologicalDescent}
 Let $C$ be a site. We say that an augmented simplicial object $p\colon X_{\bullet} \to S$ of $C$ is \defi{of cohomological descent} if the adjunction $\id \to \mathbb{R} p_* p^*$ on $\mathbb{D}_+(S)$ is an isomorphism; equivalently, $p$ is of cohomological descent if and only if the map $p^*\colon \mathbb{D}_+(S) \to \mathbb{D}_+(X_{\bullet}) $ is fully faithful \cite{Conrad:cohDescent}*{Lemma 6.8} (this explains the analogue with classical descent theory). A morphism $X \to S$ of $C$ is of cohomological descent if the associated augmented simplicial site $X_{\bullet} \to S$ is of cohomological descent (this makes sense even when $C$ does not have fiber products, since we can work in $\widetilde{C}$ instead). We say that an augmented simplicial object $X_{\bullet} \to S$ of $C$ is \defi{universally} of cohomological descent if for every $S' \to S$, the base change $X_{\bullet} \times_S S' \to S'$ (viewed in the topos $\widetilde{C}$ in case $C$ fails to admit fiber products) is of cohomological descent.

Similarly, for a sheaf of abelian groups $\calF \in \widetilde{S}$ we say that $p$ is of cohomological descent \defi{with respect to} $\calF$ if $\calF \cong \mathbb{R}p_*p^*\calF$, that a morphism $X \to S$ is of cohomological descent with respect to $\calF$ if the same is true of the associated augmented simplicial space, and universally of cohomological descent with respect to $\calF$ if for every $f\colon S' \to S$, the map $X \times_S S' \to S'$ is of cohomological descent with respect to $f^*\calF$.

Finally, we say that $p$ is of cohomological descent \defi{with respect to a subcategory} $\calC \subset \widetilde{S}$ if $p$ is of cohomological descent with respect to every $\calF \in \calC$; we say that a morphism $X \to S$ is of cohomological descent with respect to $\calC$ if the same is true of the associated augmented simplicial space, and universally of cohomological descent with respect to $\calC$ if for every $f\colon S' \to S$ and $\calF \in \calC$, the map $X \times_S S' \to S'$ is of cohomological descent with respect to $f^*\calF$.

\end{definition}

\labelpar{}
Once one knows cohomological descent for all $\calF \in \Ab \widetilde{S}$, one can deduce it for all $\calF^{\bullet} \in \mathbb{D}_+(S)$ via application of the hypercohomology spectral sequence.
\vspace{7pt}

\labelpar{}
\label{p:hypercovering}
The charm of cohomological descent is that there are interesting and useful augmented simplicial sites \emph{other than} $0$-coskeletons which are of cohomological descent. Let $C$ be a category with finite limits and let \textbf{P} be a class of morphisms in $C$ which is stable under base change and composition and contains all isomorphisms. We say that a simplicial object $X_{\bullet}$ of $C$ is a \textbf{P}-\defi{hypercovering} if for all $n \geq 0$, the natural map
\[
X_{n+1} \to (\cosk_n(\sk_n X_{\bullet}))_{n+1}
\]
is in \textbf{P}. For an augmented simplicial object $X_{\bullet} \to Y$ we say that $X_{\bullet}$ is a \textbf{P}-hypercover of $Y$ if the same condition holds for $n \geq -1$.

\begin{example}
The $0$-coskeleton $\cosk_0(X / Y) \to Y$ of a cover $X \to Y$ is a \textbf{P}-hypercover of $Y$, where \textbf{P} is the class of covering morphisms.

\end{example}

We record here many examples of morphisms of cohomological descent.

\begin{theorem}
\label{T:CDexamples}
Let $C$ be a site. Then the following are true.
\begin{itemize}
\item [(i)] Let $p\colon X \to Y$ in $\widetilde{C}$ be a covering. Then $p$ is universally of cohomological descent. Moreover, for any sheaf $F\in \Ab \widetilde{Y}$, the \v{C}ech complex $F \to p_{\bullet *}p^*_{\bullet}F$ is exact
\item [(ii)] Any morphism in $C$ which has a section locally (in $C$) is universally of cohomological descent.
\item [(iii)] The class of morphisms in $C$ universally of cohomological descent form a
topology (in the strong sense of \cite{SGA4:I}*{Expos\'e II}). In
particular, the following are true.

\begin{itemize}
\item [(a)] For a cartesian diagram of objects
\[
\xymatrix{
X' \ar[r]^{\pi'_0} \ar[d]_{f'_0} & X \ar[d]^{f_0}\\
S' \ar[r]_{\pi_0} & S
}
\]
in $C$ with $\pi_0$ universally of cohomological descent, $f_0$ is universally of cohomological descent if and only if $f_0'$ is universally of cohomological descent.

\item [(b)] If $X \to Y$ and $Y \to Z$ are maps in $C$ such that the composition $X \to Z$ is universally of cohomological descent, then so is $Y \to Z$.

\item [(c)] If $X \to Y$ and $Y \to Z$ are maps in $C$ and are universally of cohomological descent, then so is the composition $X \to Z$.
\end{itemize}

\item [(iv)] More generally, let \textbf{P} be the class of morphisms in $C$ which are universally of cohomological descent. Then a \textbf{P}-hypercover is universally of cohomological descent.

\end{itemize}

\end{theorem}

\begin{proof}

Statement (i) is \cite{Olsson:Crystal}*{Lemma 1.4.24}, (ii) follows from (i) since any morphism with a section is a covering in the canonical topology, (iii) is \cite{Conrad:cohDescent}*{Theorem 7.5}, and (iv) is \cite{Conrad:cohDescent}*{Theorem 7.10}.

\end{proof}

A mild variant applicable to a particular sheaf (as opposed to the entire category of abelian sheaves) will be useful later.

\begin{theorem}
\label{T:topologySingle}

Let $C$ be a site. Then the following are true.

\begin{itemize}
\item [(a)] Consider a cartesian diagram
\[
\xymatrix{
X' \ar[r]^{\pi'_0} \ar[d]_{f'_0} & X \ar[d]^{f_0}\\
S' \ar[r]_{\pi_0} & S
}
\]
in $C$ and let $\calF \in \widetilde S$ be a sheaf of abelian groups. Suppose $\pi_0$ is universally of cohomological descent with respect to $\calF$. Then $f_0$ is universally of cohomological descent with respect to $\calF$ if and only if $f_0'$ is universally of cohomological descent with respect to $\pi_0^*\calF$.

\item [(b)] Let $f\colon X \to Y$ and $g\colon Y \to Z$ be maps in $C$ and let $\calF \in \widetilde{Z}$ be a sheaf of abelian groups. Suppose that the composition $X \to Z$ is universally of cohomological descent with respect to $\calF$. Then $Y \to Z$ is as well.

\item [(c)] Let $f\colon X \to Y$ and $g\colon Y \to Z$ be maps in $C$ and let $\calF \in \widetilde{Z}$ be a sheaf of abelian groups. If $g$ is universally of cohomological descent with respect to $\calF$ and $f$ is universally of cohomological descent with respect to $g^*\calF$, then the composition $g \circ f$ is universally of cohomological descent with respect to $\calF$.

\item [(d)] Let $f_i \colon X_i \to Y_i$ be maps in $C$ indexed by some arbitrary set $I$. For each $i \in I$ let $\calF_i \in \widetilde{Y_i}$ be a sheaf of abelian groups. Suppose that for each $i$, $f_i$ is of cohomological descent relative to $\calF_i$. Then $\coprod f_i \colon \coprod X_i \to \coprod Y_i$ is of cohomological descent relative to $\coprod \calF_i$ (where disjoint unions are taken in $\widehat{C}$).

\end{itemize}

\end{theorem}

\begin{proof}

The proofs of (a) - (c) are identical to the proof of Theorem \ref{T:CDexamples} (iii) found in \cite{Conrad:cohDescent}*{Theorem 7.5}, and (d) follows from the fact that, setting $p_0 = \coprod f_i$, the induced morphism of simplicial topoi
\[
p\colon \widetilde{\left(\coprod X_i\right)}_{\bullet} \to \widetilde{\coprod Y_i}
\]
is also a morphism of topoi fibered over $I$, so that in particular the natural map
\[
\coprod \calF_i \to \mathbb{R}p_*p^*\coprod \calF_i
\]
is an isomorphism if and only if, for all $i \in I$, the map $\calF_i \to \mathbb{R}f_{i\bullet*}f_{i\bullet}^*\calF_i$ is an isomorphism.
\vspace{7pt}

\end{proof}

\begin{corollary}
\label{C:BCD}

Let
\[
\xymatrix{
\coprod X_i \ar[r] \ar[d] & X \ar[d]\\
\coprod Y_i \ar[r]^{\coprod v_i} & Y
}
\]
be a commutative diagram in $C$ and let $\calF \in \widetilde Y$ be a sheaf of abelian groups. Suppose that $\calF$ is universally of cohomological descent with respect to $\coprod Y_i \to Y$ and that for each $i$, $v_i^*\calF$ is universally of cohomological descent with respect to $X_i \to Y_i$. Then $\calF$ is universally of cohomological descent with respect to $X \to Y$.

\end{corollary}

\begin{proof}

This follows immediately from Theorem \ref{T:topologySingle} (b), (c), and (d).

\end{proof}

\section{The Overconvergent Site}
\label{S:definitions}

In \cite{leStum:site}, le Stum associates to a variety $X$ over a field $k$ of characteristic $p$ a ringed site $(\AN_{\g}^{\dagger}(X), \calO_{X_{\g}}^{\dagger})$ and proves an equivalence $\Mod_{\fp}(\calO_{X_{\g}}^{\dagger}) \cong \Isoc^{\dagger}(X)$ between the category of finitely presented $\calO_{X_{\g}}^{\dagger}$-modules and the category of overconvergent isocrystals on $X$. Moreover, he proves that the cohomology of a finitely-presented $\calO_{X_{\g}}^{\dagger}$-module agrees with the usual rigid cohomology of its associated overconvergent isocrystal.
\\

In this section we recall the basic definitions of \cite{leStum:site}.

\subsection{The overconvergent site}
\label{S:overconvergentSite}

Following \cite{leStum:site}, we make the following series of definitions; see \cite{leStum:site} for a more detailed discussion of the definitions with some examples.

\begin{definition}[\cite{leStum:site}, 1.2]
\label{D:overconvergentVariety}

Define an \defi{overconvergent variety} over $\calV$ to be a pair $(X \subset P, V \xrightarrow{\lambda} P_{K})$, where $X \subset P$ is a locally closed immersion of an algebraic variety $X$ over $k$ into the special fiber $P_k$ of a formal scheme $P$ (recall our convention that all formal schemes are topologically finitely presented over $\Spf \calV$), and $V \xrightarrow{\lambda} P_K$ is a morphism of analytic varieties, where $P_K$ denotes the generic fiber of $P$, which is an analytic space.  When there is no confusion we will write $(X,V)$ for $(X \subset P, V \xrightarrow{\lambda} P_{K})$ and $(X,P)$ for $(X \subset P, P_K \xrightarrow{\id} P_{K})$.
Define a \defi{formal morphism} $(X',V') \to (X,V)$ of overconvergent varieties to be a commutative diagram
\[
\xymatrix{
X' \ar@{^(->}[r] \ar[d]^f & P' \ar[d]^v & P_K' \ar[l] \ar[d]^{v_K} & V' \ar[l] \ar[d]^u\\
X \ar@{^(->}[r] & P & P_K \ar[l] & V \ar[l]
}
\]
where $f$ is a morphism of algebraic varieties, $v$ is a morphism of formal schemes, and $u$ is a morphism of analytic varieties.

Finally, define $\AN(\calV)$ to be the category whose objects are overconvergent varieties and morphisms are formal morphisms. We endow $\AN(\calV)$ with the \defi{analytic topology}, defined to be the topology generated by families $\{(X_i,V_i) \to (X,V)\}$ such that for each $i$, the maps $X_i \to X$ and $P_i \to P$ are the identity maps, $V_i$ is an open subset of $V$, and $V = \bigcup V_i$ is an open covering (recall that an open subset of an analytic space is admissible in the $G$-topology and thus also an analytic space -- this can be checked locally in the $G$-topology, and for an affinoid this is clear because there is a basis for the topology of open affinoid subdomains).

\end{definition}

\begin{definition}[\cite{leStum:site}, Section 1.1]

The specialization map $P_K \to P_k$ induces by composition a map $V \to P_k$ and we define the \defi{tube} $]X[_V$ of $X$ in $V$ to be the preimage of $X$ under this map. The tube $]X[_{P_K}$ admits the structure of an analytic space and the inclusion $i_X\colon ]X[_{P_K}\, \hookrightarrow P_K$ is a locally closed inclusion of analytic spaces (and generally not open, in contrast to the rigid case). The tube $]X[_V$ is then the fiber product $]X[_{P_K}\times_{P_K}V$ (as analytic spaces) and in particular is also an analytic space. 

\end{definition}

\begin{remark}
A formal morphism $(f,u)\colon (X', V') \to (X, V)$ induces a morphism $]f[_u\colon ]X'[_{V'} \to ]X[_V$ of tubes. Since $]f[_u$ is induced by $u$, when there is no confusion we will sometimes denote it by $u$.
\end{remark}

The fundamental topological object in rigid cohomology is the tube $]X[_{V}$, in that most notions are defined only up to neighborhoods of $]X[_{V}$. We immediately make this precise by modifying $\AN(\calV)$.

\begin{definition}[\cite{leStum:site}, Definition 1.3.3]
Define a formal morphism
\[
(f,u)\colon (X', V') \to (X, V)
\]
to be a \defi{strict neighborhood} if $f$ and $]f[_u$ are isomorphisms and $u$ induces an isomorphism from $V'$ to a neighborhood $W$ of $]X[_V$ in $V$.

\end{definition}

\begin{definition}
\label{D:overconvergentSite}

We define the category $\AN^{\dagger}(\calV)$ of \defi{overconvergent varieties} to be the localization of $\AN(\calV)$ by strict neighborhoods (which is possible by \cite{leStum:site}*{Proposition 1.3.6}): the objects of $\AN^{\dagger}(\calV)$ are the same as those of $\AN(\calV)$ and a morphism $(X',V') \to (X,V)$ in $\AN^{\dagger}(\calV)$ is a pair of formal morphisms
\[
(X',V') \leftarrow (X',W) \to (X,V),
\]
where $(X',W) \to (X',V')$ is a strict neighborhood.

The functor $\AN(\calV) \to \AN^{\dagger}(\calV)$ induces the image topology on $\AN^{\dagger}(\calV)$ (i.e., the largest topology on $\AN^{\dagger}(\calV)$ such that the map from $\AN(\calV)$ is continuous). By \cite{leStum:site}*{Proposition 1.4.1}, the image topology on $\AN^{\dagger}(\calV)$ is generated by the pretopology of collections $\{(X,V_i) \to (X,V)\}$ with $\bigcup V_i$ an open covering of a neighborhood of $]X[_V$ in $V$ and $]X[_V = \bigcup \, ]X[_{V_i}$.

\end{definition}

\begin{remark}

From now on any morphism $(X',V') \to (X,V)$ of overconvergent varieties will denote a morphism in $\AN^{\dagger}(\calV)$. One can give a down to earth description of morphisms in $\AN^{\dagger}(\calV)$ \cite{leStum:site}*{1.3.9}: to give a morphism $(X', V') \to (X,V)$, it suffices to give a neighborhood $W'$ of $]X'[_{V'}$ in $V'$ and a pair $f\colon X' \to X, u\colon W' \to V$ of morphisms which are \emph{geometrically pointwise compatible}, i.e., such that $u$ induces a map on tubes and the outer square of the diagram
\[
\xymatrix{
W' \ar[r]^{u} & V \\
]X'[_{W'} \ar[r]^{]f[_u}
\ar@{}[u]|{\bigcup |} \ar[d] & ]X[_{V}
\ar@{}[u]|{\bigcup |} \ar[d] \\
X' \ar[r]^{f} & X
}
\]
commutes (and continues to do so after any base change by any isometric extension $K'$ of $K$).

\end{remark}

\begin{definition}

For any presheaf $T \in \widehat{\AN^{\dagger}(\calV)}$, we define $\AN^{\dagger}(T)$ to be the localized category $\AN^{\dagger}(\calV)_{/T}$ whose objects are morphisms $h_{(X,V)} \to T$ (where $h_{(X,V)}$ is the presheaf associated to $(X,V)$) and morphisms are morphisms $(X',V') \to (X,V)$ which induce a commutative diagram
\[
\xymatrix{
h_{(X',V')}\ar[rr]\ar[rd] && h_{(X,V)}\ar[ld]\\
&T&
}.
\]
We may endow $\AN^{\dagger}(T)$ with the induced topology 
(i.e., the smallest topology making continuous the projection functor $\AN^{\dagger}(T) \to \AN^{\dagger}(\calV)$; see \cite{leStum:site}*{Definition 1.4.7}); concretely, the covering condition is the same as in \ref{D:overconvergentSite}. When $T = h_{(X,V)}$ we denote $\AN^{\dagger}(T)$ by $\AN^{\dagger}(X,V)$. Since the projection $\AN^{\dagger} T \to \AN^{\dagger} \calV$ is a fibered category, the projection is also cocontinuous with respect to the induced topology. Finally, an algebraic space $X$ over $k$ defines a presheaf $(X',V') \mapsto \Hom(X',X)$, and we denote the resulting site by $\AN^{\dagger}(X)$.

\end{definition}

There will be no confusion in writing $(X,V)$ for an object of $\AN^{\dagger}(T)$. 
\vspace{7pt}

We use subscripts to denote topoi and continue the above naming conventions -- i.e., we denote the category of sheaves of sets on $\AN^{\dagger}(T)$ (resp. $\AN^{\dagger}(X,V), \AN^{\dagger}(X)$) by $T_{\AN^{\dagger}}$ (resp. $(X,V)_{\AN^{\dagger}}, X_{\AN^{\dagger}}$). Any morphism $f\colon T' \to T$ of presheaves on $\AN^{\dagger}(\calV)$ induces a morphism $f_{\AN^{\dagger}}\colon T'_{\AN^{\dagger}} \to T_{\AN^{\dagger}}$ of topoi. In the case of the important example of a morphism $(f,u)\colon (X',V') \to (X,V)$ of overconvergent varieties, we denote the induced morphism of topoi by $(u^*_{\AN^{\dagger}}, u_{\AN^{\dagger}*})$.

For an analytic space $V$ we denote by $\Open V$ the category of open subsets of $V$ and by $V_{\an}$ the associated topos of sheaves of sets on $\Open V$. Recall that for an analytic variety $(X,V)$, the topology on the tube $]X[_V$ is induced by the inclusion $i_X\colon ]X[_V \, \hookrightarrow V$.

\begin{definition}[\cite{leStum:site}*{Corollary 2.1.3}]
\label{D:sheafRealization}

Let $(X,V)$ be an overconvergent variety. Then there is a morphism of sites
\[
\varphi_{X,V}\colon \AN^{\dagger}(X,V) \to \Open \, ]X[_V.
\]
The notation as usual is in the `direction' of the induced morphism of topoi and in particular backward; it is associated to the functor $\Open \, ]X[_V \to \AN^{\dagger}(X,V)$ given by $U = W\cap\, ]X[_V\,
\mapsto \, (X,W)$ (and is independent of the choice of $W$ up to strict neighborhoods). This induces a morphism of topoi
\[
(\varphi^{-1}_{X,V}, \varphi_{X,V*}) \colon (X,V)_{\AN^{\dagger}} \to (]X[_V)_{\an}.
\]

\end{definition}

\begin{definition}[\cite{leStum:site}*{2.1.7}]
Let $(X,V) \in \AN^{\dagger}(T)$ be an overconvergent variety over $T$ and let $F \in T_{\AN^{\dagger}}$ be a sheaf on $\AN^{\dagger}(T)$. We define the \defi{realization} $F_{X,V}$ of $F$ on $]X[_V$ to be $\varphi_{(X,V)*}(F|_{(X,V)_{\AN^{\dagger}}})$, where $F|_{(X,V)_{\AN^{\dagger}}}$ is the restriction of $F$ to $\AN^{\dagger}(X,V)$.

\end{definition}

We can describe the category $T_{\AN^{\dagger}}$ in terms of realizations in a manner similar to sheaves on the crystalline or lisse-\'etale sites.

\begin{proposition}[\cite{leStum:site}, Proposition 2.1.8]
\label{P:sheafReal}

Let $T$ be a presheaf on $\AN^{\dagger}(\calV)$. Then the category $T_{\AN^{\dagger}}$ is equivalent to the following category :
\begin{enumerate}

\item An object is a collection of sheaves $F_{X,V}$ on $]X[_V$ indexed by $(X, V) \in \AN^\dagger(T)$ and, for each $(f, u) \colon (X', V') \to (X, V)$, a morphism $\phi_{f,u} : ]f[_u^{-1} F_{X,V} \to F_{X', V'}$, such that as $(f,u)$ varies, the maps $\phi_{f,u}$ satisfy the usual compatibility condition.

\item A morphism is a collection of morphisms $F_{X,V} \to G_{X,V}$ compatible with the morphisms $\phi_{f,u}$.

\end{enumerate}

\end{proposition}

To obtain a richer theory we endow our topoi with sheaves of rings and study the resulting theory of modules.

\begin{definition}[\cite{leStum:site}, Definition 2.3.4]

Define the \defi{sheaf of overconvergent functions} on $\AN^{\dagger}(\calV)$ to be the presheaf of rings
\[
\mathcal O_{\mathcal V}^\dagger \colon (X, V) \mapsto \Gamma(]X[_V, i_X^{-1}\mathcal
O_V)
\]
where $i_X$ is the inclusion of $]X[_V$ into $V$; this is a sheaf by \cite{leStum:site}*{Corollary 2.3.3}. For $T \in \widehat{\AN^{\dagger}(\calV)}$ a presheaf on $\AN^{\dagger}(\calV)$, define $\calO^{\dagger}_{T}$ to be the restriction of $\calO^{\dagger}_{\calV}$ to $\AN^{\dagger}(T)$.

\end{definition}

We follow our naming conventions above, for instance denoting by $\calO^{\dagger}_{(X,V)}$ the restriction of $\calO^{\dagger}_{\calV}$ to $\AN(X,V)$.

\begin{remark}
\label{R:ringedRealization}

By \cite{leStum:site}*{Proposition 2.3.5, (i)}, the morphism of topoi of Definition \ref{D:sheafRealization} can be promoted to a morphism of ringed sites
\[
(\varphi_{X,V}^*, \varphi_{X,V*}) \colon (\textrm{AN}^\dagger(X,V), \mathcal O_{(X,V)}^\dagger) \to (]X[_V, i_X^{-1}\mathcal O_V).
\]
In particular, for $(X,V) \in \AN^{\dagger} T$ and $M \in \calO^{\dagger}_T$, the realization $M_{X,V}$ is an $i^{-1}_{X}\calO_V$-module. For any morphism $(f,u)\colon (X',V') \to (X,V)$ in $\AN^{\dagger}(T)$, one has a map
\[
(]f[_{u}^\dagger, ]f[_{u*}) \colon (]X'[_{V'}, i_{X',V'}^{-1} \mathcal O_{V'}) \to (]X[_V, i_{X,V}^{-1} \mathcal O_V).
\]
of ringed sites, and functoriality gives transition maps
\[
\phi^{\dagger}_{f,u}\colon ]f[_{u}^{\dagger}M_{X,V} \to M_{X',V'}
\]
which satisfy the usual cocycle compatibilities.

\end{remark}

We can promote the description of $T_{\AN^{\dagger}}$ in Proposition \ref{P:sheafReal} to descriptions of the categories $\Mod \calO^{\dagger}_{T}$ of $\calO^{\dagger}_{T}$-modules, $\QCoh \calO^{\dagger}_{T}$ of quasi-coherent $\calO^{\dagger}_{T}$-modules (i.e., modules which locally have a presentation), and $\Mod_{\fp} \calO^{\dagger}_{T}$ of locally finitely presented $\calO^{\dagger}_{T}$-modules.

\begin{proposition}[\cite{leStum:site}, Proposition 2.3.6]
\label{P:moduleReal}

Let $T$ be a presheaf on $\AN^{\dagger}(\calV)$. Then the category $\Mod \calO^{\dagger}_T$ (resp. $\QCoh \calO^{\dagger}_T$, $\Mod_{\fp} \calO^{\dagger}_T$) is equivalent to the following category :
\begin{enumerate}

\item An object is a collection of sheaves $M_{X,V} \in \Mod i^{-1}_X \calO_V$ (resp. $\QCoh i^{-1}_X\calO_V$, $\Coh i^{-1}_X\calO_V$) on $]X[_V$ indexed by $(X, V) \in \AN^\dagger(T)$ and, for each $(f, u) \colon (X', V') \to (X, V)$, a morphism (resp. isomorphism) $\phi^{\dagger}_{f,u}\colon ]f[_u^{\dagger}M_{X,V} \to M_{X', V'}$, such that as $(f,u)$ varies, the maps $\phi^{\dagger}_{f,u}$ satisfy the usual compatibility condition.

\item A morphism is a collection of morphisms $M_{X,V} \to M'_{X,V}$ compatible with the morphisms $\phi_{f,u}^{\dagger}$.

\end{enumerate}

\end{proposition}

\begin{definition}[\cite{leStum:site}, Definition 2.3.7]

Define the \defi{category of overconvergent crystals on $T$}, denoted $\Cris^{\dagger} T$, to be the full subcategory of $\Mod \calO^{\dagger}_T$ such that the transition maps $\phi_{f,u}^{\dagger}$ are isomorphisms.

\end{definition}

\begin{example}

The sheaf $\calO^{\dagger}_{T}$ is a crystal, and in fact $\QCoh \calO^{\dagger}_T \subset \Cris^{\dagger} T$.

\end{example}

\begin{remark}
\label{R:crystalEquivalence}

It follows immediately from the definition of the pair $(\varphi_{X,V}^*, \varphi_{X,V*})$ of functors that $\varphi_{X,V*}$ of a $\calO^{\dagger}_{(X,V)}$-module is a crystal, and that the adjunction $\varphi_{X,V}^* \varphi_{X,V*}E \to E$ is an isomorphism if $E$ is a crystal. If follows that the pair $\varphi_{X,V}^*$ and $\varphi_{X,V*}$ induce an equivalence of categories
\[
\Cris^{\dagger} (X,V) \to \Mod i_X^{-1}\mathcal O_V;
\]
see \cite{leStum:site}*{Proposition 2.3.8} for more detail.

\end{remark}

One minor subtlety is the choice of an overconvergent variety as a base.

\begin{definition}
\label{D:overconvergentBase}

Let $(C,O) \in \AN^{\dagger}(\calV)$ be an overconvergent variety and let $T \to C$ be a morphism from a presheaf on $\Sch_k$ to $C$. Then $T$ defines a presheaf on $\AN^{\dagger}(C,O)$ which sends $(X,V) \to (C,O)$ to $\Hom_C(X,T)$, which we denote by $T/O$. We denote the associated site by $\AN^{\dagger}(T/O)$, and when $(C,O) = (S_k,S)$ for some formal $\calV$-scheme $S$ we write instead $\AN^{\dagger}(T/S)$.

\end{definition}

The minor subtlety is that there is no morphism $T \to h_{(C,O)}$ of presheaves on $\AN^{\dagger}(\calV)$. A key construction is the following.

\begin{definition}[\cite{leStum:site}*{Paragraph after Corollary 1.4.15}]
\label{D:imagePresheaf}

Let $(X,V) \to (C,O) \in \AN^{\dagger}(\calV)$ be a morphism of overconvergent varieties. We denote by $X_V/O$ the image presheaf of the morphism $(X,V) \to X/O$, considered as a morphism of presheaves. Explicitly, a morphism $(X',V') \to X/O$ lifts to a morphism $(X',V') \to X_V/O$ if and only if there exists a morphism $(X',V') \to (X,V)$ over $X/O$, and in particular different lifts $(X',V') \to (X,V)$ give rise to the same morphism $(X',V') \to X_V/O$. When $(C,O) = (\Spec k, \calM(K))$, we may write $X_V$ instead $X_V/\calM(K)$.

\end{definition}

Many theorems will require the following extra assumption of \cite{leStum:site}*{Definition 1.5.10}. Recall that a morphism of formal schemes $P' \to P$ is said to be proper at a subscheme $X \subset P'_k$ if, for every component $Y$ of $\overline{X}$, the map $Y \to P_k$ is proper (see \cite{leStum:site}*{Definition 1.1.5}).

\begin{definition}
\label{D:realization}

Let $(C,O) \in \AN^{\dagger}(\calV)$ be an overconvergent variety and let $f\colon X \to C$ be a morphism of $k$-schemes. We say that a formal morphism $(f,u)\colon (X,V) \to (C,O)$, written as
\[
\xymatrix{
X \ar@{^(->}[r] \ar[d]^f & P \ar[d]^v & V\ar[d]^u \ar[l]\\
C \ar@{^(->}[r] & Q & O\ar[l]
},
\]
is a \defi{geometric realization} of $f$ if $v$ is proper at $X$, $v$ is smooth in a neighborhood of $X$, and $V$ is a neighborhood of $]X[_{P_K \times_{Q_K} O}$ in $P_K \times_{Q_K} O$. We say that $f$ is \defi{realizable} if there exists a geometric realization of $f$.

\end{definition}

\begin{example}

Let $Q$ be a formal scheme and let $C$ be a closed subscheme of $Q$. Then any projective morphism $X \to C$ is realizable.

\end{example}

We need a final refinement to $\AN^{\dagger}(\calV)$.

\begin{definition}
\label{D:good}

We say that an overconvergent variety $(X,V)$ is \defi{good} if there is a good neighborhood $V'$ of $]X[_V$ in $V$ (i.e., every point of $]X[_V$ has an affinoid neighborhood in $V$). We say that a formal scheme $S$ is good if the overconvergent variety $(S_k,S_K)$ is good. We define the \defi{good overconvergent site} $\AN^{\dagger}_{\g}(T)$ to be the full subcategory of $\AN^{\dagger}(T)$ consisting of good overconvergent varieties. Given a presheaf $T \in \AN^{\dagger}(\calV)$, we denote by $T_{\g}$ the restriction of $T$ to $\AN^{\dagger}_{\g}(\calV)$.

\end{definition}

Note that localization commutes with passage to good variants of our sites (e.g., there is an isomorphism $\AN^{\dagger}_{\g}(\calV)_{/T_{\g}} \cong \AN^{\dagger}_{\g}(T)$). When making further definitions we will often omit the generalization to $\AN^{\dagger}_{\g}$ when it is clear.
\vspace{7pt}

The following proposition will allow us to deduce facts about $\Mod_{\fp} \calO^{\dagger}_{X_g}$ from results about $(X,V)$ and $X_V$.

\begin{proposition}
\label{P:coverings}

Let $(C,O) \in \AN^{\dagger}_{\g}(\calV)$ be a good overconvergent variety and let $(X,V) \to (C,O)$ be a geometric realization of a morphism $X \to C$ of schemes. Then the following are true:
\begin{itemize}
\item [(i)] The map $(X,V)_{\g} \to (X/O)_{\g}$ is a covering in $\AN^{\dagger}_{\g}(\calV)$.
\item [(ii)] There is an equivalence of topoi $(X_V/O)_{\AN^{\dagger}_g} \cong (X/O)_{\AN^{\dagger}_g}$.
\item [(iii)] The natural pullback map $\Cris_{\g}^{\dagger} X/O \to \Cris_{\g}^{\dagger} X_{V}/O$ is an equivalence of categories.
\item [(iv)] Suppose that $(X,V)$ is good. Then the natural map $\Cris^{\dagger} X_V/O \to \Cris_{\g}^{\dagger} X_V/O$ is an equivalence of categories.
\end{itemize}

\end{proposition}

\begin{proof}

The first two claims are \cite{leStum:site}*{1.5.14, 1.5.15}, the third follows from the second, and the last is clear.

\end{proof}

In particular, the natural map $\Mod_{\fp} \calO^{\dagger}_{X_{\g}} \to \Mod_{\fp} \calO^{\dagger}_{(X_V)_{\g}} \cong \Mod_{\fp} \calO^{\dagger}_{X_V}$ is an equivalence of categories.
\vspace{7pt}

\subsection{Technical lemmas}
\label{S:siteLemmas}

We state here a few technical lemmas that will be useful in the proof of Theorem \ref{T:mainCohDescentTheoremPrelude}.

\begin{lemma}
\label{L:baseChange}

Let $(Y,W) \to (X,V)$ be a morphism of overconvergent varieties. Let $Y' = Y\times_X X'$ and $W' = W \times_V X'$. Then $(Y',W') \cong (Y,W)\times_{(X,V)} (Y',V')$.

\end{lemma}

\begin{proof}

This is the comment after \cite{leStum:site}*{Proposition 1.3.10}.

\end{proof}

\begin{lemma}
\label{L:closedExactness}

Let $p \colon (X',V') \to (X,V)$ be a morphism of overconvergent
varieties such that the induced map on tubes is an inclusion of a
closed subset. Then $p_*$ is exact.

\end{lemma}

\begin{proof}

It suffices to check that, for any cartesian diagram
\[
\xymatrix{
(Y', W') \ar[r] \ar[d]^{p'}
& \ar[d] (X', V')
\\
(Y, W) \ar[r]
& (X,V)
}
\]
the map induced on tubes by $p'$ is exact; the lemma follows since for
any base change of $p$, the induced map on tubes is also an inclusion of a
closed subset and such maps are exact.

\end{proof}

\section{Cohomological descent for overconvergent crystals}
\label{S:overCohDescent}

In this section we prove Theorem \ref{T:mainCohDescentTheoremPrelude}. The proof naturally breaks into cases: Zariski covers, modifications, finitely presented flat covers, and proper surjections. The full proof fails
without the goodness assumption, but many special cases (e.g., cohomological descent with respect to Zariski hypercovers) hold without the goodness assumption.

\subsection{Zariski Covers}
\label{S:etaleCohDescent}

We begin with the case of a Zariski cover. One can restate the main result of \cite{leStum:site}*{Section 3.6} as the statement that a Zariski covering is universally of cohomological descent (see Definition \ref{D:cohomologicalDescent}) with respect to crystals. Throughout this subsection we omit distinction between $\AN^{\dagger}$ and $\AN^{\dagger}_g$, but remark here that each result is true for either site.
\vspace{4pt}

Some care is needed to interpret le Stum's results in the language of cohomological descent; to that end, we first prove a few lemmas that will be useful in later proofs as well.

\begin{lemma}
\label{L:tubeIsoDescent}

Let $p_0 =(f_0,u_0) \colon (X_0,V_0) \to (X_{-1},V_{-1})$ be a morphism of overconvergent varieties such that
\begin{itemize}
\item [(i)] the induced map $\,]f_0[ \,\colon \, ]X_0[_{V_0} \to
\,]X_{-1}[_{V_{-1}}$ is an isomorphism, and
\item [(ii)] the natural map $]f_0[^{-1} i_{X_{-1}}^{-1}\calO_{V_{-1}}
\to i_{X_0}^{-1}\calO_{V_{0}} $
is an isomorphism.
\end{itemize}
Then $p_0$ is universally of cohomological descent with respect to crystals.

\end{lemma}

\begin{remark}
\label{R:quasiImmersionIsOK}

Note that condition (ii) of Lemma \ref{L:tubeIsoDescent} is satisfied if $u_0$ is a finite quasi-immersion and thus in particular is satisfied if $u_0$ is an isomorphism or if $V_0$ is a neighborhood of $\,]X_{-1}[_{V_{-1}}$. (Note that these are non-trivial conditions, since $f_0$ may not be an isomorphism.) Moreover, condition (i) holds if $f_0$ is surjective.

We also note that condition (ii) is necessary; in general, a morphism
\[
(X,]X[_V) \to (X,V)
\]
is not universally of cohomological descent for crystals, since the
\v{C}ech complex is not exact. For example, when $X = \A^1$ and $V =
\P^1_K$, condition (ii) fails, and indeed the \v{C}ech complex
\[
0 \to K\{t\}^{\dagger} \to K\{t\} \xrightarrow{0} K\{t\} \to \ldots
\]
is not exact.

\end{remark}

\begin{proof}

We check the hypotheses of Corollary \ref{C:exactImpliesCD} (ii). Denote by $p_i \colon(X_i,V_i) \to (X_{-1},V_{-1})$ the $i+1$ fold fiber product of the map $p_0\colon (X_0, V_0) \to (X_{-1},V_{-1})$. Noting that formation of tubes commutes with base change (and in particular that the map on tubes induced by $p_i$ is an isomorphism), it follows from Lemma \ref{L:closedExactness} that $p_{i,*}$ is exact.

For each $i \geq 0$ and $j > 0$, each projection $p_i^j\colon (X_i,V_i) \to (X_{i-1},V_{i-1})$ also induces an isomorphism $]X_i[_{V_i}\cong \, ]X_{i-1}[_{V_{i-1}} $ on tubes. Moreover, for a fixed $i$, the maps $p_i^j$ are all equal. Finally, note that by condition (ii), the natural maps $\calF \to p_{i,*}^jp_i^{j,*}\calF$ are all isomorphisms (contrast with Remark \ref{R:quasiImmersionIsOK}). It follows that the \v{C}ech complex $\calF \to p_*p^*\calF$ is exact (since the maps alternate between an isomorphism and the zero map). The lemma follows.

\end{proof}

\begin{lemma}
\label{L:openTubeDescent}

Let $\{(X_i,V_i) \to (X,V)\}$ be a collection of morphisms of overconvergent varieties such that each $V_i \to V$ is an open immersion and $\{\, ]X_i[_{V_i}\}$ is an open covering of $\,]X[_V$. Then
the map
\[
u_0\colon \coprod (X_i,V_i) \to (X,V)
\]
is universally of cohomological descent with respect to crystals.

\end{lemma}

\begin{proof}

Let $W_i'$ be an open subset of $V$ such that $W_i'\, \cap \, ]X[_V = \, ]X_i[_{V_i}$ (which exists since $]X_i[_{V_i}$ is an open subset of $]X[_V$). Let $W_i$ be the preimage of $W_i'$ under the map $V_i \to V$. By Corollary \ref{T:topologySingle} (b) it suffices to prove that the map $u_0'\colon \coprod (X_i,W_i) \to (X,V)$ is universally of cohomological descent with respect to crystals.

Let $v_i$ denote the morphism $(X_i,\, W_i) \to (X,\, V)$. Then the morphism $u'_0$ factors as
\[
\coprod (X_i,\, W_i) \xrightarrow{\coprod{v_i}}
\coprod (X,\, W_i) \xrightarrow{w_0}
(X,\, V).
\]
The map $w_0$ is a covering in $\AN^{\dagger}(X,V)$ and thus universally of cohomological descent by \ref{T:CDexamples} (i). Since, by the construction of $W_i$, the induced map $]X_i[_{W_i} \to \, ]X[_{W_i}$ is an isomorphism, it follows from Lemma \ref{L:tubeIsoDescent} that each map $v_i$ is universally of cohomological descent with respect to crystals. By Theorem \ref{T:topologySingle} (d), $\coprod v_i$ is universally of cohomological descent with respect to crystals; by \ref{T:topologySingle} (c), the composition is also universally of cohomological descent with respect to crystals and the lemma follows.

\end{proof}

\begin{definition}
We say that a collection $\{X_i\}$ of subspaces of a topological space $X$ is a \defi{locally finite covering} if $X = \cup X_i$ and if each point $x$ of $X$ admits an open neighborhood $U_x$ on which $\{X_i \cap\, U_x\}$ admits a
finite refinement which covers $U_x$.

\end{definition}

\begin{lemma}
\label{L:closedTubeDescent}

Let $\{(f_i,u_i) \colon (X_i,V_i) \to (X,V)\}$ be a collection of morphisms of overconvergent varieties such that
\begin{itemize}
\item[(a)] the maps $]u_i[\colon ]X_i[_{V_i} \to \, ]X[_V$ are closed inclusions of topological spaces,
\item[(b)] $\{\,]X_i[_{V_i}\}$ is a locally finite covering of $\, ]X[_V$, and
\item [(c)] for each $i$, the natural map $]f_i[^{-1} i_{X}^{-1}\calO_{V}
\to i_{X_i}^{-1}\calO_{V_{i}} $
is an isomorphism.
\end{itemize}
Then the map
\[
p \colon \coprod (X_i,V_i) \to (X,V)
\]
is universally of cohomological descent with respect to crystals.

\end{lemma}

\begin{proof}

Let $F \in \Cris^{\dagger}(X,V)$ be a crystal. By Corollary \ref{C:exactImpliesCD}, it suffices to prove that (i) $\mathbb{R}^qp_{\bullet,*}p_{\bullet}^*F = 0$ for $q > 0$, and (ii) the \v{C}ech complex $F \to p_{\bullet,*}p_{\bullet}^*F$ is exact. Let $p_j$ be the $j$-fold fiber product of $p$. By the spectral sequence (Equation \ref{spectralSequence}), it suffices to prove that $\mathbb{R}^qp_{j*}p_j^*F = 0$ for $ j \geq 0$ and $q > 0$; since for each $j$, $p_j$ is a disjoint union of maps which induce closed inclusions on tubes, this follows from Lemma \ref{L:closedExactness}.

For (ii), it suffices to check that, for every map $\pi\colon (X',V') \to (X,V)$, the realization with respect to $(X',V')$ of the \v{C}ech complex of $\pi^{-1}F$ with respect to
\[
p' \colon \coprod (X'_i,V'_i) \to (X',V')
\]
(where $X'_i = X_i \times_X X'$ and $V_i' = V_i \times_V V'$) is exact, which (noting that our hypotheses are stable under base change) since the tubes form a locally finite closed covering, follows from condition (c) and the proof of \cite{leStum:site}*{Proposition 3.1.4}.

\end{proof}

\begin{corollary}
\label{C:formalLocal}

Let $(X \hookrightarrow P \leftarrow V)$ be an overconvergent variety and let $\{P_i\}_{i \in I}$ be a collection of Zariski open formal subschemes of $P$. Let $(X_i,V) = (X\times_P P_i \hookrightarrow P \leftarrow V)$ and let $(X_i,V_i) = (X\times_P P_i \hookrightarrow P_i \leftarrow V\times_{P_K}(P_i)_K)$. Suppose that $\{X_i\}$ forms a locally finite Zariski open cover of $X$. Then the following are true.
\begin{itemize}
\item[(1)] The map
$\coprod (X_i, V) \to (X,V)$ is universally of cohomological descent with respect to crystals.
\item[(2)] Suppose $(X,V)$ is good and that the tubes
$\{\,]X_i[_{V_i}\}$ cover a neighborhood of $\,]X[_V$ in $V$.
Then $\coprod (X_i, V_i) \to (X,V)$ is universally of
cohomological descent with respect to finitely presented crystals.
\end{itemize}

\end{corollary}

\begin{remark}
By Remark \ref{R:quasiImmersionIsOK}, the extra hypothesis on the tubes in claim (2) is necessary.
\end{remark}

\begin{proof}

Since specialization is anti-continuous, the tubes form a locally finite closed covering and claim (1) thus follows from Lemma \ref{L:closedTubeDescent}. For claim (2), since $(X,V)$ is good we may assume that $V$ is affinoid. The claim then follows by Tate's Acylicity Theorem \cite{BGR}*{8.2, Corollary 5}.

\end{proof}

We say that a morphism of schemes $X \to Y$ over $k$ is universally of cohomological descent (resp., with respect to a sheaf $\calF \in \Ab(Y_{\AN^{\dagger}})$) if the associated morphism $X_{\AN^{\dagger}} \to Y_{\AN^{\dagger}}$ is universally of cohomological descent (resp., with respect to $\calF$).

\begin{theorem}
\label{T:zariskiDescent}

Let $(C,O)$ be an overconvergent variety and let $X \to C$ be a morphism of algebraic varieties. Let $\{U_i\}_{i \in I}$ be a locally finite covering of $X$ by open subschemes (resp., a covering of $X$ by closed subschemes) and denote by $\alpha_0\colon U = \coprod_{i \in I} U_i \to X$ the induced morphism of schemes. Denote by $\alpha\colon U_{\bullet} \to X$ the 0-coskeleton of $\alpha_0$. Then the morphism of topoi $U_{\bullet}/O_{\AN^{\dagger}} \to X/O_{\AN^{\dagger}}$ is universally of cohomological descent with respect to $\calF$.

\end{theorem}

\begin{proof}

The proof for $\alpha_{\AN^{\dagger}_{\g}}$ is identical to the proof for $\alpha_{\AN^{\dagger}}$. We note that the map $\coprod (X',V) \to X$, where the coproduct is taken over $\AN^{\dagger}(X/O)$, is a covering in the canonical topology on $\AN^{\dagger}(X/O)$ and thus universally of cohomological descent. Setting $U_i' = X' \times_X U_i$, the diagram (of sheaves on $\AN^{\dagger} O$)
\[
\xymatrix{
\coprod_{\AN^{\dagger}(X/O)} \coprod_i (U'_i, V) \ar[r] \ar[d]
& \ar[d] \coprod_i U_i
\\
\coprod_{\AN^{\dagger}(X/O)} (X', V) \ar[r]
& X
}
\]
commutes. By Lemma \ref{L:closedTubeDescent} (resp., Lemma \ref{L:openTubeDescent}) the maps $\coprod (U'_i,V) \to (X',V)$ are universally of cohomological descent with respect to crystals; the theorem thus follows from Corollary \ref{C:BCD}.

\end{proof}

\begin{remark}
\label{R:schemeVsheafUnion}

Let $\{X_i\}$ be a collection of schemes. Then the presheaf on $\AN^{\dagger} \calV$ represented by the disjoint union $\coprod X_i$ (as schemes) is \emph{not} equal to the disjoint union (as presheaves) of the presheaves represented by each $X_i$. Nonetheless, Theorem \ref{T:topologySingle} (d) also holds for the map in $\AN^{\dagger} \calV$ represented by a disjoint union $\coprod Y_i \to \coprod X_i$ of morphisms of schemes (taken as a disjoint union of schemes instead of as presheaves on $\AN^{\dagger} \calV$); indeed, the sheafification of $\coprod X_i$ is the same in each case, and in general for a site $C$ and a presheaf $F \in \widehat{C}$ with sheafification $F^a$, there is a natural equivalence
\[
\widetilde{C_{/F}} \cong \widetilde{C}_{/F^a}
\]
of topoi.

\end{remark}

\begin{corollary}
\label{C:zariskiLocal}

Let $(C,O)$ be an overconvergent variety and let $X \to Y$ be a morphism of algebraic varieties over $C$. Let $\{Y_i\}$ be a locally finite open cover of $Y$ and denote by $X_i$ the fiber product $X \times_Y Y_i$. Then $X/O_{\AN^{\dagger}} \to Y/O_{\AN^{\dagger}}$ is universally of cohomological descent with respect to crystals if and only if for each $i$, the map $X_i/O_{\AN^{\dagger}} \to Y_i/O_{\AN^{\dagger}}$ is universally of cohomological descent with respect to crystals.

\end{corollary}

\begin{proof}

This follows from Theorems \ref{T:zariskiDescent} and \ref{T:topologySingle} (a) and (d) applied to the diagram of sheaves on $\AN^{\dagger} (Y/O)$ induced by the cartesian diagram of schemes
\[
\xymatrix{
\coprod X_i \ar[r] \ar[d]
& \ar[d] X
\\
\coprod Y_i \ar[r]
& Y
}.
\]

\end{proof}

The following direct corollary to Theorem \ref{T:zariskiDescent} allows us to reduce to the integral case.

\begin{corollary}
\label{C:components}

Let $Y$ be an algebraic variety. Let $\{Y_i'\}$ be the set of irreducible components of $Y$ and let $Y_i := (Y_i')_{\text{red}}$ be the reduction of $Y_i'$. Then the morphism $\coprod Y_i \to Y$ is universally of cohomological descent with respect to crystals.

\end{corollary}

\subsection{Modifications}
\label{ss:blowups}

In order to apply Raynaud-Gruson's theorem on `flattening stratifications', we now address cohomological descent for modifications. The following lemma is a translation of \cite{Tsuzuki:rigidDescent}*{Lemma 3.4.5}, with a minor variation in that we work with non-archimedean analytic spaces. Note also that some care (e.g., the use of Lemma \ref{L:tubeIsoDescent}) is necessary to apply his argument to the overconvergent site.

\begin{lemma}[\cite{Tsuzuki:rigidDescent}*{Lemma 3.4.5}]
\label{L:CDblowUp}

Let $Y$ be a scheme and let $Z$ be a closed subscheme whose sheaf of ideals $I$ is generated by two elements $f$ and $g$. Then the blow up $X \to Y$ of $Y$ with respect to $I$ is universally of cohomological descent with respect to crystals.

\end{lemma}

\begin{proof}

By Corollary \ref{C:zariskiLocal}, we may assume that $Y$ is affine and thus admits an embedding $Y \hookrightarrow P$ into a formal scheme $P$ which is proper over $\calV$. Let $\{P_i\}$ be a locally finite open affine cover of $P$ and let $Y_i = Y\times_P P_i$. Let $\overline{Y_i}$ be the closure of $Y_i$ in $P_i$.

Let $f_i$ (resp. $g_i$) denote the restriction of $f$ (resp. $g$) to $Y_i$ and denote by $Z_i$ the subscheme defined by $f_i$ and $g_i$. Choose lifts $\overline{f_i}$ and $\overline{g_i}$ of $f_i$ and $g_i$ to $\Gamma(\calO_{\overline{Y_i}})$ and define $\overline{Z_i}$ be the subscheme of $\overline{Y_i}$ defined by $\overline{f_i}$ and $\overline{g_i}$. Let $\overline{X_i}$ be the blow up of $\overline{Y_i}$ along $\overline{Z_i}$. Then $\overline{Z_i} \times_{\overline{Y_i}} Y_i = Z_i$ and $\overline{X_i} \times_{\overline{Y_i}} Y_i = X_i$.

Let $U_{i,1}$ be the tube $]\overline{Z_i}[_{(P_i)_K}$ of $\overline{Z_i}$ in $(P_i)_K$. Fix a rational number $\lambda$ in $(0,1)$, let $\hat{f_i}$ and $\hat{g_i}$ be lifts of $\overline{f_i}$ and $\overline{g_i}$ to $\Gamma(\calO_{P_i})$, and define
\[
U_{i,2} = \{x \in \, ]\overline{X_i}[_{(P_i)_K} : |\hat{f_i}| > \lambda \text{ or } |\hat{g_i}| > \lambda \};
\]
by construction $U_{i,1} \cup U_{i,2}$ is a cover of $]\overline{X_i}[_{(P_i)_K}$.

The scheme $\overline{X_i}$ is a subscheme of $\P^1_{P_i}$ (indeed, if $s$ and $t$ are coordinates for $\P^1$, then $\overline{X_i}$ is defined by the equation $\overline{f_i}t - \overline{g_i}s$). Set $V_{i,1} := U_{i,1} \times_{(P_i)_K} (\P^1_{P_i})_K \cong \P^1_{U_{i,1}}$. The map $(X_i,V_{i,1}) \to (Y_i,U_{i,1})$ of overconvergent varieties factors as $(X_i,V_{i,1}) \to (Y_i,V_{i,1})\to (Y_i,U_{i,1})$. The second map has a section and is thus universally of cohomological descent by Theorem \ref{T:CDexamples} (ii), and the first map is universally of cohomological descent with respect to crystals by Lemma \ref{L:tubeIsoDescent} (for the first map, note that since $X_i \to Y_i$ is surjective, the map on tubes is an isomorphism); we conclude that $(X_i,V_{i,1}) \to (Y_i,U_{i,1})$ universally of cohomological descent with respect to crystals by Theorem \ref{T:topologySingle} (c).

Let $R_i$ be the closed formal subscheme of $\P^1_{P_i}$ defined by the equation $\hat{f}_it - \hat{g}_is$. Then $(R_i)_K \to (P_i)_K$ is an isomorphism away from the vanishing locus of $\hat{f_i}$ and $\hat{g_i}$ in $(P_i)_K$. Denote by $V_{i,2}$ the pre-image of $U_{i,2}$ under the map $(R_i)_K \to (P_i)_K$. Then the map $(X_i,V_{i,2}) \to (Y_i,U_{i,2})$ of overconvergent varieties factors as $(X_i,V_{i,2}) \to (Y_i,V_{i,2})\to (Y_i,U_{i,2})$; the second map is an isomorphism, and the first map is universally of cohomological descent with respect to crystals by Lemma \ref{L:tubeIsoDescent} (again, since $X_i \to Y_i$ is surjective, the map on tubes is an isomorphism); we conclude that $(X_i,V_{i,2}) \to (Y_i,U_{i,2})$ is universally of cohomological descent with respect to crystals by Theorem \ref{T:topologySingle} (c).

We get a diagram
\[
\xymatrix{
\coprod\left( (X_i,V_{i,1})\coprod(X_i,V_{i,2}) \right)\ar[d]\ar[rrr]&&&X \ar[d]\\
\coprod\left( (Y_i,U_{i,1})\coprod(Y_i,U_{i,2}) \right)\ar[r] &\coprod (Y_i,(P_i)_K) \ar[r] &(Y,P_K) \ar[r]&Y
}.
\]
The middle horizontal map is universally of cohomological descent with respect to crystals by Lemma \ref{C:formalLocal} and the left and right horizontal maps are universally of cohomological descent by Theorems \ref{P:coverings} and \ref{T:CDexamples} (i); thus the composition is universally of cohomological descent with respect to crystals by Theorem \ref{T:topologySingle} (c). By the previous two paragraphs and Theorem \ref{T:topologySingle} (d), the left vertical map is universally of cohomological descent with respect to crystals; the lemma thus follows from Theorem \ref{T:topologySingle} (b).

\end{proof}

The next lemma lets us reduce the case of a general blow up to the situation of Lemma \ref{L:CDblowUp}.

\begin{lemma}
\label{L:blowupStructureTheorem}

Let $Y$ be a Noetherian integral scheme, let $I \subset \calO_Y$ be a sheaf of ideals globally generated by $r \geq 2$ many elements and let $X \to Y$ be the blow up of $Y$ along $I$. Then there exists a map $X' \to X$ such that the composition $X' \to Y$ factors as
\[
X' = X_{r'} \to X_{r'-1} \to \cdots \to X_i \to X_{i-1} \to \cdots \to
X_0 = Y,
\]
where each map $X_i \to X_{i-1}$ is a blow up centered at an ideal which is globally generated by two elements.

\end{lemma}

\begin{proof}

This is a special case of \cite{Tsuzuki:rigidDescent}*{Lemma 3.4.4}.

\end{proof}

Recall that a morphism $p\colon X \to Y$ is a \defi{modification} if it is an isomorphism over a dense open subscheme of $Y$. The next proposition shows that modifications are universally of cohomological descent with respect to crystals.

\begin{proposition}
\label{P:CDmodification}
Let $p\colon X \to Y$ be a modification. Then $p$ is universally of cohomological descent with respect to crystals.

\end{proposition}

\begin{proof}

By Chow's lemma \cite{EGAII}*{Theorem 5.6.1} and \ref{T:topologySingle} (b), we may assume that $p$ is projective. By Corollary \ref{C:components} and Theorem \ref{T:topologySingle} (b) and (c) we may assume that $Y$ is integral, and then by \cite{liu:algebraicGeometry}*{Section 8, Theorem 1.24}, there exists an affine open cover $\{Y_i\}$ of $Y$ such that for each $i$, $Y_i\times_YX \to Y_i$ is a blow up of $Y_i$ along a closed subscheme; by Corollary \ref{C:zariskiLocal} we may thus assume that $p$ is a blow up. By the structure lemma for blow ups (Lemma \ref{L:blowupStructureTheorem}) we may reduce to the case of a codimension one blow up which is Lemma \ref{L:CDblowUp}.

\end{proof}

\subsection{Flat Covers}
\label{S:flatOverCohDescent}

In this section we prove Theorem \ref{T:mainCohDescentTheoremPrelude} (i) -- that finitely presented crystals are universally cohomologically descendable with respect to fppf (faithfully flat locally finitely presented) morphisms of schemes.

\begin{definition}
\label{D:finiteOverconvergent}

A map $(f,u) \colon (X',V') \to (X,V)$ of overconvergent varieties is said to be \defi{finite} (see \cite{leStum:site}*{Definition 3.2.3}) if, up to strict neighborhoods, $u$ is finite (see \cite{Berkovich:nonArchEtaleCoh}*{paragraph after Lemma 1.3.7}) and $u^{-1}(]X[_V) = \, ]X'[_{V'}$. Moreover, $u$ is said to be \defi{universally flat} if $u$ is quasi-finite and, locally for Grothendieck topology on $V'$ and $V$, $u$ is of the form $\calM(A') \to \calM(A)$ with $A \to A'$ flat (see \cite{Berkovich:nonArchEtaleCoh}*{Definition 3.2.5}).

\end{definition}

\begin{proposition}
\label{P:finiteDescent}

Let $(f,u)\colon (X',V') \to (X,V)$ be a finite map of overconvergent varieties and suppose that, after possibly shrinking $V'$ and $V$, $u$ is universally flat and surjective. Then $(f,u)$ is universally of cohomological descent with respect to finitely presented overconvergent crystals.

\end{proposition}

\begin{proof}

To ease notation we set $p := (f,u)$. Let $F \in \Mod_{\fp}(X,V)$. By Corollary \ref{C:exactImpliesCD}, it suffices to prove that (i) $\mathbb{R}^qp_{\bullet,*}p_{\bullet}^*F = 0$ for $q > 0$, and (ii) the \v{C}ech complex $F \to p_{\bullet,*}p_{\bullet}^*F$ is exact. Let $p_i := (f_i,u_i)\colon (X_i,V_i) \to (X,V)$ be the $i$-fold fiber product of $p$; $p_i$ also satisfies the hypotheses of this proposition. By the spectral sequence (Equation \ref{spectralSequence}), it suffices to prove that $\mathbb{R}^qp_{i*}p_i^*F = 0$ for $ i \geq 0$ and $q > 0$.

Shrink $V$ and $V_i$ such that $u_i$ is finite and such that $F_{X,V}$ is isomorphic to $i_X^{-1}G$ for some $G \in \Coh \calO_V$ (which is possible by \cite{leStum:site}*{Proposition 2.2.10}). To prove (i), one can work with realizations as in \cite{leStum:site}*{Proof of Proposition 3.2.4}; it thus suffices to prove that $\mathbb{R}^q]u_i[_*]u_i[^*F_{X,V} = 0$ for $q > 0$. Then $\mathbb{R}^q]u_i[_*]u_i[^*F_{X,V} = i_X^{-1}\mathbb{R}^qu_{i*}u_i^*G$; by \cite{Berkovich:nonArchEtaleCoh}*{Corollary 4.3.2} $\mathbb{R}^qu_{i*}u_i^*G = 0$ and (i) follows. 

For (ii), since one can check exactness of a complex of abelian sheaves on the collection of all good realizations and since our hypotheses are stable under base change, it suffices to prove that the \v{C}ech complex of $F_{X,V}$ with respect to $]u[$ is exact. Since $i_X^{-1}$ is exact, it suffices to prove that the \v{C}ech complex of $G$ with respect to $u$ is exact.

By \cite{Berkovich:nonArchEtaleCoh}*{Proposition 4.1.2}, $G$ is a sheaf in the flat quasi-finite topology, so by Theorem \ref{T:CDexamples} (i), $G \to u_{\bullet,*}u^*_{\bullet}G$ is exact in the flat quasi-finite topology; since $G$ is coherent and $(X,V)$ is good, this is exact in the usual topology.


\end{proof}

 Recall that a \defi{monogenic} map of rings is a map of the form $A \to   A[t]/f(t)$, where $f \in A[t]$ is a monic polynomial, and a map of   affine formal schemes is said to be monogenic if the associated map   on rings is monogenic.

\begin{proof}[Proof of Theorem \ref{T:mainCohDescentTheoremPrelude} (i)]

By Theorem \ref{T:zariskiDescent} and Corollary \ref{C:BCD}, we may assume that everything is affine.
\vspace{4pt}

Step 0: (Reduction to the finite and locally free case). Let $p\colon X \to Y$ be an fppf cover. By \cite{stacks-project}*{Lemma \href{http://math.columbia.edu/algebraic_geometry/stacks-git/locate.php?tag=05WN}{05WN}}, there exists a map $X' \to X$ such that the composition $X' \to Y$ is a composition of surjective finite locally free morphisms and Zariski coverings; by Theorem \ref{T:zariskiDescent} and Corollary \ref{C:BCD}, we may assume that $X \to Y$ is finite and locally free.
\vspace{5pt}

Step 1: (Monogenic case). Suppose that $X \to Y$ is monogenic and choose a closed embedding $Y \hookrightarrow \A^n_{\calV}$ (which exists since $Y$ is affine and of finite type) and then an open immersion $\A^n_{\calV}  \subset P := {\mathbb{P}}^n_{\calV}$. The polynomial defining $X \to Y$ lifts to a monic polynomial with coefficients, giving a monogenic (and thus finite and flat) map $X_0 \to \A^n_{\calV}$, and then homogenizing this polynomial gives a map $\pi\colon P' \to P$ of schemes over $\Spec \calV$ and an embedding $X \hookrightarrow P'$ which is compatible with the embedding $Y \hookrightarrow P$. The map $\pi$ may not be finite or flat (see Remark \ref{R:notFlat} below), but (noting that $\pi$ is projective) by \cite{GrusonR:flatification}*{Th\'eor\`em 5.2.2}, there exists a modification $\widetilde{P} \to P$, centered away from $X$,  such that the strict transform $\widetilde{P'} \to \widetilde{P}$ of $P' \to P$ is flat and (since it is generically finite, flat, and proper) finite.

Replacing $P' \to P$ with the formal completion of $\widetilde{P'} \to \widetilde{P}$, we thus have a finite flat map $P' \to P$ of formal schemes and an embedding $X \hookrightarrow P'$ which is compatible with the embedding $Y \hookrightarrow P$. Consider the diagram
\[
\xymatrix{
(X, P'_K) \ar[r]\ar[d] &
X\ar[d] \\
(Y, P_K) \ar[r]&
Y}.
\]
By Theorems \ref{P:coverings} and \ref{T:CDexamples} (i), $(Y, P_K) \to Y$ is universally of cohomological descent with respect to crystals, so by Corollary \ref{C:BCD} it suffices to prove that $(X,P'_K) \to (Y,P_K)$ is universally of cohomological descent with respect to finitely presented crystals. Since $X = Y \times_P P'$, $(X,P'_K) \to (Y,P_K)$ satisfies the hypotheses of Proposition \ref{P:finiteDescent} and step 1 follows.
\vspace{7pt}

Step 2: (Base extension). Now let $k \subset k'$ be a finite field extension of the residue
field. We claim that it suffices to check that $X_{k'} \to Y_{k'}$ is universally of cohomological descent with respect to finitely
presented crystals. Indeed, let $k = k_0 \subset k_1 \subset \ldots \subset
k_{n-1} \subset k_n = k'$ be a sequence of field extensions such that
$k_i = k_{i-1}(\alpha_i)$ for some $\alpha_i \in k_i$ (note that one may not be able to choose $n
= 1$ since $k \subset k'$ may not be separable). Consider the diagram
\[
\xymatrix
{
X_{k_n} \ar[rrrr] \ar[d] &&&& X \ar[d] \\
Y_{k_n} \ar[r] & Y_{k_{n-1}} \ar[r] & \cdots \ar[r] & Y_{k_{1}} \ar[r]
&Y_{k_0}
}.
\]
Each map $Y_{k_i} \to Y_{k_{i-1}}$ is monogenic and thus universally of cohomological descent with respect to finitely presented crystals, the claim thus follows from Theorem \ref{T:topologySingle} (b), (c), and (d).
\vspace{7pt}

Step 3: (Reduction to the monogenic case). Let $\kbar$ be the algebraic closure of $k$ and let $p_{\kbar}\colon X_{\kbar} \to Y_{\kbar}$ be the base change of $p$ to $\kbar$. Let $x \in X_{\kbar}$ be a closed point and set $y = p_{\kbar}(x)$. Let $\kbar(x)$ and $\kbar(y)$ denote the residue fields of $x$ and $y$; since $\kbar$ is algebraically closed, $\kbar(x) = \kbar(y)$. In particular, $\kbar(y)$ is a separable extension of $\kbar(x)$, and thus, by the argument of \cite{BoschLR:Neron}*{2.3, Proposition 3}, there exists a (generally non-cartesian) commutative diagram
\[\xymatrix{
X_x \ar[r] \ar[d]& X_{\kbar} \ar[d] \\
Y_y \ar[r] & Y_{\kbar}
}
\]
where $X_x$ (resp. $Y_y$) is an affine open neighborhood of $x$ (resp. $y$) and $X_x \to Y_y$ is monogenic. By quasi-compactness of $Y_{\kbar}$, there thus exists a (generally non-cartesian) commutative diagram
\[\xymatrix{
\coprod X_i \ar[r] \ar[d]^{\coprod f_i} & X_{\kbar} \ar[d] \\
\coprod Y_i \ar[r] & Y_{\kbar}
}
\]
such that $\{Y_i\}$ is a finite cover of $Y_{\kbar}$ by affine open subschemes of finite type over $\kbar$, $X_i$ is an affine open subscheme of $X_{\kbar}$, and each map $f_i\colon X_i \to Y_i$ is monogenic. Since the covering is finite, there exists a finite field extension $k \subset k'$ and a (generally non-cartesian) commutative diagram
\[\xymatrix{
\coprod X'_i \ar[r] \ar[d]^{\coprod f_i} & X_{k'} \ar[d] \\
\coprod Y'_i \ar[r] & Y_{k'}
}
\]
with the same properties. By step 2, it suffices to check that $X_{k'} \to Y_{k'}$ is universally of cohomological descent with respect to finitely presented crystals. By Corollary \ref{C:BCD}, it suffices to prove this for each $i$, the map $X'_i \to Y'_i$, which follows from step 1.

\end{proof}

\begin{remark}
\label{R:notFlat}
The modification in step 1 of the proof is necessary. Indeed, the monogenic map $X \to \A^2$ given by $t^2 + x_1x_2t + x_1 + x_2$ homogenizes to the map $X' \to \P^2$ given by $x_0^2t^2 + x_1x_2ts + (x_1 + x_2)x_0s^2$ which is not flat, since it is generically quasi-finite but not quasi-finite (since the fiber over $x_0 = x_1 = 0$ is $\P^1$).

\end{remark}

\subsection{Proper surjections}
\label{S:properOverCohDescent}

The proper case of the main theorem will now follow from Chow's lemma and the Raynaud-Gruson theorem on `Flattening Blow Ups'.

\begin{proof}[Proof of Theorem \ref{T:mainCohDescentTheoremPrelude} (ii)]
\label{p:mainTheoremProper}

Let $p\colon X \to Y$ be a proper surjection. By Chow's lemma \cite{EGAII}*{Theorem 5.6.1} and Theorem \ref{T:topologySingle} (b), we may assume that $p$ is projective. By Corollary \ref{C:zariskiLocal} we may assume that $Y$ is affine and thus by \cite{GrusonR:flatification}*{Th\'eor\`em 5.2.2}, there exists a modification $Y' \to Y$ such that the strict transform $X' \to Y'$ is flat. By Theorem \ref{T:mainCohDescentTheoremPrelude} (i) (resp. \ref{P:CDmodification}) $X' \to Y'$ (resp. $Y' \to Y$) is universally of cohomological descent with respect to finitely presented crystals. By \ref{T:topologySingle} (c), the composition $X' \to Y$ is universally of cohomological descent, and the proper case of the main theorem follows from \ref{T:topologySingle} (b).

\end{proof}


\bibliographystyle{alpha}

\end{document}